\documentclass[12pt]{article}

\usepackage{amssymb,latexsym,amsmath,amsthm,amsfonts,vmargin}
\usepackage{upgreek}
\usepackage[]{todonotes}
\usepackage{ulem}
\usepackage[hidelinks]{hyperref}
\usepackage[capitalise, nameinlink]{cleveref}
\usepackage{soul}
\usetikzlibrary{patterns, decorations.markings,arrows, calc}

\newtheorem{theorem}{Theorem}

\newtheorem{lemma}[theorem]{Lemma}
\newtheorem{proposition}[theorem]{Proposition}
\newtheorem{corollary}[theorem]{Corollary}

\theoremstyle{definition}
\newtheorem{definition}[theorem]{Definition}

\newtheorem{remark}[theorem]{Remark}
\newtheorem{example}[theorem]{Example}

\numberwithin{equation}{section}
\numberwithin{theorem}{section}

\DeclareMathOperator{\diag}{diag}
\DeclareMathOperator{\diam}{diam}

\DeclareMathOperator{\Res}{Res}
\DeclareMathOperator{\supp}{supp}

\newcommand{\ds}{\displaystyle}
\newcommand{\wt}{\widetilde}

\newcommand{\lozengeblue}{
	w_{\begin{tikzpicture}[scale=0.2]
			\draw[fill=blue!20] (0,0) rectangle (1,1);
\end{tikzpicture}}}
\newcommand{\lozengeyellow}{
	w_{\begin{tikzpicture}[scale=0.2]
			\draw[fill=yellow!50] (0,0) -- (1,1) -- (1,2) -- (0,1) -- cycle;
\end{tikzpicture}}}
\newcommand{\lozengered}{
	w_{\begin{tikzpicture}[scale=0.2]
			\draw[fill=red!20] (0,0) -- (1,0) -- (2,1) -- (1,1) -- cycle; 
\end{tikzpicture}}}

\let\Re\undefined
\let\Im\undefined
\DeclareMathOperator{\Re}{Re}
\DeclareMathOperator{\Im}{Im}

\title{Critical measures on higher genus Riemann surfaces}

\author{Marco Bertola
	\footnote{Department of Mathematics and Statistics, Concordia
	University, 1455 de Maisonneuve W., Montr\'eal, Qu\'ebec, 
Canada H3G 1M8, Email: marco.bertola@concordia.ca, and
SISSA, International School for Advanced Studies, via Bonomea 265, Trieste, Italy. Email: Marco.Bertola@sissa.it.  Supported in part by the Natural Sciences and Engineering Research Council of Canada (NSERC) grant RGPIN-2016-06660. }
	\and
	Alan Groot 
	\footnote{Department of Mathematics, Katholieke Universiteit Leuven, Leuven, Belgium, Email: alan.groot@kuleuven.be, alangroot@gmail.com. 
		Supported by
		long term structural funding-Methusalem grant of the Flemish
		Government.}	 
	\and  Arno B.J. Kuijlaars
	\footnote{Department of Mathematics,  Katholieke Universiteit Leuven, Leuven, Belgium, Email: arno.kuijlaars@kuleuven.be. Supported by long term structural funding-Methusalem grant of the Flemish
		Government, and by FWO Flanders projects
		EOS 30889451 and 
		G.0910.20.}	  
}

\begin{document}
	\maketitle
	
	\begin{abstract}
	Critical measures in the complex plane are saddle points 
	for the logarithmic energy with external field. Their local and
	global structure was described by Mart\'inez-Finkelshtein and Rakhmanov.
	In this paper we start the development of a theory of critical measures on higher genus Riemann surfaces, where the logarithmic 
	energy is replaced by the energy with respect to 
	a bipolar Green's kernel.
	We study a max-min problem for the bipolar Green's energy
	with external fields $\Re V$ where $dV$ is a
	meromorphic differential.  
	Under reasonable assumptions the max-min problem has a
	solution and we show that the corresponding equilibrium 
	measure is a critical measure in the external field. 
	In a special genus one situation we are able to show 
	that the critical
	measure is supported on maximal trajectories of a 
	meromorphic quadratic differential.
	
	We are motivated by applications to random lozenge tilings 
	of a hexagon with periodic weightings. 
	Correlations in these models are expressible in terms of 
	matrix valued orthogonal polynomials. The matrix 
	orthogonality is interpreted as (partial) scalar orthogonality 
	on a Riemann surface. The theory of critical measures  
	will be useful for the asymptotic analysis of a corresponding
	Riemann-Hilbert problem as we outline in the paper.
		
	\end{abstract}
			
	\section{Introduction}
	
	The notion of critical measures in the complex plane was
	developed by Mart\'inez-Finkelshtein and Rakhmanov
	\cite{MFR11,MFR16,Rak12} with the aim of studying asymptotic
	zero distributions of Heine-Stieltjes polynomials. It is related
	to the asymptotics of orthogonal polynomials with a non-hermitian
	orthogonality on contours in the complex plane as initiated by
	Stahl \cite{Sta86,Sta91} and Gonchar and Rakhmanov \cite{GR87}. 
	The goal of this paper is to generalize
	the notion of critical measures to higher genus Riemann 
	surfaces. Our motivation to do so comes from the analysis of certain
	tiling problems with periodic weightings as we will explain in Section~\ref{sec:motivation} of this paper.
	
	\subsection{Critical measure in the complex plane}
	We start by recalling the notion of critical measure 
	in the plane following \cite{MFR11}. It can be defined
	in a more general situation, but we restrict here to the
	case of an external field $\Re V$ where $V'$ is a rational
	function on $\mathbb C$. In fact, $V$ itself can be multi-valued,
	but $\Re V$ is assumed to be well-defined and
	single-valued on $\mathbb C$. The logarithmic energy of
	a probability measure $\mu$ in the external field $\varphi = \Re V$ is
	\begin{equation} \label{EVmu} 
		E_\varphi\left[\mu\right] =  \iint \log \frac{1}{|s-t|} d\mu(s) d\mu(t)
		+ \int \varphi d\mu. 
	\end{equation}
	
	A critical measure is a probability measure $\mu$ such that
	\eqref{EVmu} is stationary with respect to certain
	perturbations of $\mu$, known as Schiffer variations.
	For a continuous function $h: \mathbb C \to \mathbb C$, and a probability measure $\mu$, 
	the one parameter family $(\mu_{\varepsilon,h})_{\varepsilon \in \mathbb R}$
	of probability measures is defined through their action on 
	continuous functions $f$, 
	\begin{equation} \label{muepsh} 
		\int f d\mu_{\varepsilon,h} = 
		\int f\left(s+\varepsilon h(s)\right) d\mu(s),
		\qquad \varepsilon \in \mathbb R.  \end{equation}
	
	\begin{definition} \label{definition11}
		The probability measure $\mu$ is a \textbf{critical measure} 
		in the external	field $\varphi$ if 
		\begin{equation} \label{DVh} 
			\lim_{\varepsilon \to 0} \frac{
				E_\varphi\left[ \mu_{\varepsilon,h} \right] - E_\varphi\left[\mu\right]}{\varepsilon} = 0 \end{equation}
		for every $C^1$ function $h$  with compact support.
	\end{definition} 
	The definition is a special case of \cite[Definition~3.2]{MFR11}.
	By \cite[Lemma~3.1]{MFR11}, the limit \eqref{DVh} exists 
	and is equal to 
	\[ - \Re \left[ \iint \frac{h(s)-h(t)}{s-t} d\mu(s) d\mu(t) - 
	\int h(s) V'(s) d\mu(s) \right]. \]
	By considering both $h$ and $ih$, it follows that $\mu$
	is a critical measure if and only if 	
	\begin{equation} \label{CM1}
		\iint \frac{h(s)-h(t)}{s-t} d\mu(s) d\mu(t)
		=  \int h(s) V'(s) d\mu(s) 
	\end{equation}
	for every $C^1$ function $h$ with compact support.
	With a limiting argument \eqref{CM1} can be extended
	to $C^1$ functions $h$ that are bounded on the support of $\mu$.
	
	The main result on critical measures is that they 
	are supported on trajectories of quadratic differentials. 
	Recall that  a trajectory
	of the quadratic differential $-Q dz^2$ is an (open or closed)
	contour $\Sigma$ such that $Q(z(s)) (z'(s))^2 < 0$
	where $s \mapsto z(s)$ is any smooth parametrization of $\Sigma$,
	see e.g.\ \cite{Str84}.
	
	\begin{theorem}[Mart\'inez-Finkelshtein and Rakhmanov \cite{MFR11}] \label{theo:MFR}  Suppose $\varphi = \Re V$ where $V'$ is rational 
	on $\mathbb C$.
	For a critical measure $\mu$ in the external field $\varphi$, 
		we let
		\begin{equation} \label{QD3} 
			Q(z) 
			= \left(\frac{V'(z)}{2} \right)^2 - \int \frac{V'(z)-V'(s)}{z-s} d\mu(s).  \end{equation}
		Then the following hold.
		\begin{enumerate}
			\item[\rm (a)] 
			$Q$ is a rational function with the property that
			\begin{equation} \label{QD2} \left[\int \frac{d\mu(s)}{z-s} - \frac{V'(z)}{2} \right]^2 = Q(z),
				\quad m_2 \text{-a.e.},  \end{equation}
			where $m_2$ denotes the planar Lebesgue measure.
			\item[\rm (b)] 
			The support $\Sigma = \supp(\mu)$ consists of a 
			finite union of maximal trajectories of the quadratic differential $-Q(s) ds^2$,
			and on each trajectory we have (with $ds$ the
			complex line element and an appropriate branch of the square root)
			\[ d\mu(s) = \frac{1}{\pi i} Q(s)^{1/2} ds, \qquad s \in \Sigma.\]
			\item[\rm (c)] The logarithmic potential 
			$U^{\mu}(z) =	\ds \int \log \frac{1}{|z-s|}d\mu(s)$ of $\mu$
			satisfies
			\begin{equation} \label{Sprop1} 
				2 U^{\mu}(z) + \Re V(z) = c_j, \qquad z \in \Sigma_j, 
			\end{equation}
			with a constant $c_j$ that can be different for
			each connected component $\Sigma_j$ of $\Sigma$.
			\item[\rm (d)]  
			Any  point $z \in \Sigma$ that is not a zero of $Q$
			has a neighborhood $D$ such that $D \cap \Sigma$
			is an analytic arc, and  
			\begin{equation} \label{Sprop2} 
				\frac{\partial}{\partial n_+} \left(2U^{\mu} + \Re V\right)(z)
				= \frac{\partial}{\partial n_-} \left(2 U^{\mu} + \Re V\right)(z), \quad z \in D \cap \Sigma,
			\end{equation} 
			where $\frac{\partial}{\partial n_{\pm}}$ denote the two
			normal derivatives to $\Sigma$ at $z$.
		\end{enumerate}
	\end{theorem}
	\begin{proof} See Theorem~5.1, Lemma~5.3 and Lemma~5.4 in
		\cite{MFR11}.
	\end{proof}

The identity \eqref{Sprop2} is known as the $S$-property
in the external field $\Re V$ and $\Sigma$ is called an $S$-contour or $S$-curve
in the external field $\Re V$. Together with \eqref{Sprop1}, it
implies that the $g$-function
\[ g(z) = \int \log(z-s) d\mu(s) \]
satisfies
\begin{equation} \label{Sprop3} 
	g_+(z) + g_-(z) - V(z) = \ell_j \qquad \text{ on } \Sigma_j 
\end{equation}
for a complex constant $\ell_j$ that can be different
for each connected component $\Sigma_j$ of $\Sigma$. Indeed, the
real part of $g_+ + g_- - V$ is constant on each
component by \eqref{Sprop1}, while \eqref{Sprop2} implies
that the imaginary part is constant on each component as well,
as can be seen from an application of the Cauchy-Riemann equations. 

\subsection{Equilibrium measures}

For a compact set $F \subset \mathbb C$ the equilibrium
measure in the external field $\varphi = \Re V$ is the probability
measure on $F$ that minimizes the functional \eqref{EVmu} among 
probability measures on $F$, see e.g.\ \cite{Dei99, Ran95, ST97}. 
If $F$ is a continuum (i.e.,
a compact connected set with more than one point),
or a union of continua, and $\Re V$ is bounded from
below on $F$, then there is a unique 
equilibrium measure $\mu^F$ in the external field $\Re V$. 
It satisfies
\begin{align}  \label{EMplane1}
	2 U^{\mu^F}(z) + \Re V(z) & = c,  \quad \text{ on } \supp(\mu^F), \\ \label{EMplane2}
	2 U^{\mu^F}(z) + \Re V(z) & \geq c, \quad \text{ on } F,
\end{align} 
for some constant $c$. We emphasize that $c$ is the same
on all connected components of $\supp(\mu^F)$, which is
in contrast to the situation in Theorem~\ref{theo:MFR} (c).
For a critical measure the constant can be different on 
each connected component of its support.
We are interested in equilibrium measures that are also 
critical measures, and, given $V$, this will depend on the 
choice of a good compact set $F$.

To determine such $F$, Kuijlaars and Silva \cite{KS15} introduced the notion 
of a critical set, based on the {\it extremal compact} considered by Rakhmanov in \cite{Rak12}. To describe it, we use 
\begin{equation} \label{criticalF1}
	E_\varphi(F) = E_\varphi[\mu^F]. 
\end{equation}

\begin{definition}
	Let $F$ be a compact set such that $\varphi = \Re V$ is
	bounded from below on $F$.
	Then $F$ is a \textbf{critical set} in the external field $\varphi$ if
	\begin{equation} \label{criticalF2} \lim_{\varepsilon \to 0} \frac{E_\varphi(F_{\varepsilon,h}) - E_\varphi(F)}{\varepsilon} = 0 
	\end{equation}
	for every $C^1$ function $h$ with compact support, 
	where  
	\begin{equation} \label{criticalF3} 
		F_{\varepsilon,h} = \{ x + \varepsilon h(x) \mid x \in F\}. 
	\end{equation}
\end{definition}
Note that \eqref{criticalF3} is the support of the deformed
measure $\mu_{\varepsilon,h}$ if $\supp(\mu) = F$.

Rakhmanov \cite{Rak12} essentially proved the following, see also \cite{KS15}.
\begin{proposition} \label{prop:CS}
	Let $F$ be a union of continua such that $\varphi = \Re V$ is
	bounded from below on $F$. Let $\mu^F$ be its equilibrium
	measure in the external field $\varphi$.
	\begin{enumerate}
		\item[\rm (a)] Then the limit in \eqref{criticalF2}
		exists and it is equal to the limit in \eqref{DVh}, i.e.,
		\[ \lim_{\varepsilon \to 0} \frac{E_\varphi(F_{\varepsilon,h}) - E_\varphi(F)}{\varepsilon}
		= \lim_{\varepsilon \to 0}  
		\frac{E_\varphi[\mu^{F}_{\varepsilon,h}] - E_\varphi[\mu^F]}{\varepsilon}. \]
		\item[\rm (b)] $F$ is a critical set if and only if 
		$\mu^F$ is a critical measure in the external field.
		\item[\rm (c)] Suppose $F$ belongs to a family $\mathcal F$
		of union of continua such that 
		for every $C^1$ function $h$ there is $\varepsilon_0 >0$	such that
		$F_{\varepsilon,h} \in \mathcal F$ for every $\varepsilon \in
		(-\varepsilon_0, \varepsilon_0)$, where $F_{\varepsilon,h}$ is given by \eqref{criticalF3}. Suppose also that  
		\begin{equation} \label{criticalF4}
			E_\varphi(F) = \max_{F'\in \mathcal F} E_\varphi(F'). 
		\end{equation}
		Then $F$ is a critical set and $\mu^F$ is a critical measure
		in the external field $\varphi$.
	\end{enumerate}
\end{proposition}
\begin{proof} The proof of part (a) follows from the property that  
	\begin{equation} \label{CSisCM} E_\varphi(F_{\varepsilon,h}) - E_\varphi[\mu^F_{\varepsilon,h}] = o(\varepsilon) \quad
		\text{ as } \varepsilon \to 0, \end{equation}
	see \cite[Section 4]{KS15} or \cite[Section 9.10]{Rak12} for details, and see also 
	Proposition~\ref{prop:CShigher} below where the analogous
	statement is proved in the higher genus case.
	Part (b) is immediate from part (a) and the definitions of
	critical measure and critical set. 
	If $F$ satisfies the conditions of part (c), then 
	$E_\varphi(F_{\varepsilon, h})$ has a local maximum at 
	$\varepsilon = 0$ for every $C^1$ function $h$. 
	Hence \eqref{criticalF2} holds and $F$ is a critical set.
	Then also $\mu^F$ is a critical measure by part (b).
\end{proof}
Part (c) shows that one may find critical sets from solving a
max-min energy problem
\[ \max_{F \in \mathcal F}  \min_{\supp(\mu) \subset F}  E_\varphi[\mu] \]
where the minimum is over probability measures on $F$ and
the maximum is over a suitable family $\mathcal F$ satisfying
the condition of part (c) of Proposition~\ref{prop:CS}.
If a maximizer $F$ exists then clearly \eqref{criticalF4}
holds and $\mu^F$ is a critical measure in the external field $\varphi$. Then Theorem~\ref{theo:MFR} applies and it follows
in particular that the support of $\mu^F$ is a union of
maximal trajectories of a quadratic differential.  

It can be shown that under suitable conditions, a maximizer indeed exists.
This program was first carried out by Kamvissis and Rakhmanov \cite{KR05} in the case of weighted Green's energy in the upper half plane, and later by Rakhmanov \cite{Rak12} in the case of weighted logarithmic energy in the complex plane; see also the work of Kuijlaars and Silva \cite{KS15} for the case of a polynomial $V$.
Our extension to higher genus (see Theorem~\ref{thm:residues} below and its
proof in Section~\ref{sec:maxmin}) follows this approach.

There is a substantial literature on determining 
critical equilibrium measures in the external field $\Re V$
where $V$ is a polynomial. This is motivated in part by
questions in random matrix theory, see the recent papers
\cite{BBDY22,BBGMT22,BGM21} and references cited therein.  
See also \cite{MFS19} for an extension to vector equilibrium measures.

\section{Potential theory on Riemann surfaces}
\label{sec:general-genus}

Before we can state the results of this paper we need
to introduce certain notions from potential theory on a higher genus 
Riemann surface, see also Skinner \cite{Ski15} and 
a recent series of papers by Chirka \cite{Chi18,Chi19,Chi20}.
Our exposition will focus on equilibrium measures in an external field.
 
Throughout, we use $X$ to denote a compact Riemann surface with a distinguished point $p_\infty$, which 
we refer to as the point at infinity.

\subsection{Bipolar Green's function} 

To extend Theorem~\ref{theo:MFR} we first of all need
the appropriate  analogue of the logarithmic kernel
that appears in \eqref{EVmu} to define the logarithmic
energy. This is provided by the bipolar Green's function.

\begin{proposition} \label{prop:bipolarGreen}
	Let $X$ be a compact Riemann surface with a
	distinguished point $p_{\infty} \in X$. 
	There is a function $(p,q) \mapsto G(p,q)$ defined
	on $X \times X$ such that
	\begin{enumerate}
		\item[\rm (a)] for a fixed $q \in X \setminus \{p_{\infty} \}$
		the function $p \mapsto G(p,q)$ is real-valued and harmonic
		on $X \setminus \{ p_{\infty}, q\}$,
		\item[\rm (b)] if $z$ is a local coordinate at $q$, then
		\begin{equation}  \label{Green2}
			G(p,q) = -\log|z(p)| + O(1) \quad \text{ as } p \to q, \end{equation}
		\item[\rm (c)] if $z_{\infty}$ is a local coordinate at
		$p_{\infty}$, then 
		\begin{equation} \label{Green3} 
			G(p,q) =  \log |z_{\infty}(p)| + O(1)
			\quad \text{ as } p \to p_{\infty}, \end{equation}
		\item[\rm (d)] $G(p,q) = G(q,p)$.
	\end{enumerate}
\end{proposition}
As usual, a local coordinate at a point on a Riemann surface 
means a holomorphic chart in which the point corresponds to 
$0 \in \mathbb C$.

\begin{proof}
	The existence of $G$  satisfying
	parts (a), (b) and (c) of Proposition~\ref{prop:bipolarGreen}
	can be found in Gamelin \cite{Gam01}, Simon \cite[Section 3.8]{Sim15}. We could not find an appropriate reference for part (d), although it may be known.
	The parts (a), (b) and (c) determine $p \mapsto G(p,q)$ up to an additive
	constant that may depend on $q$. It is noted by 
	Skinner \cite[p.~25]{Ski15}
	that it is not immediately obvious how to choose the constant
	such that $G$ is symmetric. 
	We give a proof in the appendix.
\end{proof}

\begin{definition} \label{def:bipolarG}
	The function $G$ satisfying
	the conditions of Proposition~\ref{prop:bipolarGreen}
	is called the \textbf{bipolar Green's function} 
	with one pole at $p_{\infty}$, or simply 
	bipolar Green's function.
\end{definition}
The bipolar Green's function is unique
up to an additive constant, but the constant will not be 
important for us. Note that $G(p,q) \to -\infty$ as 
$p \to p_{\infty}$ and $G(p,q) \to +\infty$ as $p \to q$.

\begin{example} \label{example23}
	(a)
	If $X$ is the Riemann sphere with $p_{\infty}$ the
	point at infinity, then $G(p,q) = \log \frac{1}{|p-q|}$.
	\medskip
	
	(b)
	On a complex torus $X = \mathbb C \slash \Lambda$
	with lattice $\Lambda = \mathbb Z + \tau \mathbb Z$ and
	$\Im \tau > 0$,
	the bipolar Green's function with pole
	at $p_{\infty} = 0$ (modulo $\Lambda$) is explicitly
	given by, see e.g.~\cite{Chi18},
	\cite[Section 2]{KM17}, \cite{Ski15},
	\begin{equation} \label{Green1} 
		G(p,q) = \log \left| \frac{\theta_1(p) \theta_1(q)}{\theta_1(p-q)} \right|
		- \frac{2\pi}{\Im \tau} \left(\Im p\right) \left(\Im q\right) 
	\end{equation}
	in terms of the Jacobi elliptic function $\theta_1$ that has a zero 	at $0$ (see \cite[Chapter 20]{DLMF}
	which however uses a different scaling of elliptic functions 
	with periods $\pi$ and $\pi \tau$
	instead of $1$ and $\tau$), i.e.,
	\begin{equation} \label{Jacobi1}
		\theta_1(z) = \theta_1(z;\tau)
		= - i \sum_{k=-\infty}^{\infty}
		(-1)^k e^{\pi i \tau (k+\frac{1}{2})^2 +(2k+1)\pi i z}.
	\end{equation}
	The following properties of $\theta_1$ can be found in \cite[Chapter 20]{DLMF}. 
	The Jacobi elliptic function is an odd entire function
	with simple zeros at every lattice point, and no other zeros.
	Moreover, it has the quasi-periodicity properties
	\begin{equation} \label{Jacobi1periods} 
		\theta_1(z+1) = - \theta_1(z),
	\qquad \theta_1(z+\tau) =  - e^{-\pi i \tau - 2\pi iz} \theta_1(z). 
	\end{equation}
	
	We can use \eqref{Green1} to construct the bipolar Green's function on an arbitrary genus one Riemann surface $X$,
	by composing it with the Abel map from $X$ to $\mathbb C \slash \Lambda$
	that maps $p_{\infty}$ to $0$.
\end{example}

The local behavior of $G(p,q)$ is determined by 
the logarithmic kernel $\log \frac{1}{|p-q|}$ in the following
sense.

\begin{lemma} We have
	\begin{enumerate}
		\item[\rm (a)] 	If $z_{\infty}$ is a local coordinate at $p_\infty$
		(that maps $p_{\infty}$ to $0$) 	then we have 
		\begin{equation} \label{Gnearinfty} 
			G(p,q) - \log \frac{1}{|z_{\infty}(p)^{-1}-z_{\infty}(q)^{-1}|} =  O(1)
		\end{equation}
		uniformly for $p,q$ in a neighborhood of $p_\infty$.
		\item[\rm (b)] 	If $z_0$ is a local coordinate at $p_0 \neq p_{\infty}$ (that maps $p_0$ to $0$) 
		then we have
		\begin{equation} \label{Gnearp0} 
			G(p,q) - \log \frac{1}{|z_0(p)-z_0(q)|} = O(1) 
		\end{equation}
		uniformly for $p,q$ in a neighborhood of $p_0$ that does not contain $p_\infty$.
	\end{enumerate}
\end{lemma}
\begin{proof}
	The two terms in the left hand sides of \eqref{Gnearinfty} 
	and \eqref{Gnearp0}
	have the same logarithmic singularities when $p=q$.
	The two terms in \eqref{Gnearinfty} also have
	the same logarithmic singularity when $p= p_{\infty}$
	and by symmetry of $G$ also when $q=p_{\infty}$. 
	Therefore the differences in both \eqref{Gnearinfty} and
	\eqref{Gnearp0} are harmonic in both variables 
	near $p_{\infty}$ and $p_0$, respectively,
	and therefore they remain bounded.
\end{proof}

\subsection{Equilibrium measure in external field}\label{sec:eqmeas}

The bipolar Green's function allows us to extend  concepts of logarithmic potential theory in the complex plane \cite{Dei99,Ran95,ST97} 
to compact Riemann surfaces.

\begin{definition} 
	Let $\mu$ be a measure on $X$ with  
	compact support in $X \setminus \{p_{\infty}\}$.
	Then we define its \textbf{bipolar Green's energy}
	\begin{equation} \label{GreenE0}
		E \left[\mu \right] = \iint G(p,q)d\mu(p) d\mu(q).
	\end{equation}
	For a lower semi-continuous $\varphi : \supp(\mu) \to \mathbb R \cup \{+\infty\}$ we define the
	bipolar Green's energy of $\mu$ in the external field $\varphi$ by
	\begin{equation} \label{GreenE1} 
		E_\varphi\left[\mu\right] = \iint G(p,q) d\mu(p) d\mu(q) 
		+ \int \varphi d\mu. \end{equation}
\end{definition}

We extend the definition~\eqref{GreenE0} 
to signed measures $\nu = \mu_1 -\mu_2$, provided $E[\mu_1]$ 
and $E[\mu_2]$ are finite. 
We need the crucial property
\begin{proposition} \label{prop:Gposdef}
	If $\mu_1, \mu_2$ are two compactly supported 
	probability measures on $X
	\setminus \{p_{\infty}\}$
	with finite bipolar Green's energy and $\nu = \mu_1 - \mu_2$, then
	\begin{equation} \label{Gposdef} 
		E[\nu] = \iint G(p,q) d\nu(p) d\nu(q) \geq 0, \end{equation}
	and $E[\nu]= 0$ if and only if $\mu_1 = \mu_2$.
\end{proposition}	
\begin{proof}
	A number of proofs are known in the literature. One proof
	involves the eigenfunctions of the Laplace operator
	as in \cite{Chi19} and \cite{Ski15}. Another proof
	relates  \eqref{Gposdef} to the norm
	of certain functions in a Sobolev space; see \cite{Chi19} and \cite{KM17}.
\end{proof}

Let $\mathcal E^1(K)$ denote the set of all probability
measures $\mu$ on $K$ with $E[\mu] < +\infty$.
If $\mu \in \mathcal E^1(K)$ and $(\mu_n)_n$ is a sequence 
in $\mathcal E^1(K)$ such that
\begin{equation}\label{energynorm}
	\lim_{n \to \infty} E[\mu_n-\mu] = 0,
\end{equation} 
then we say that $(\mu_n)_n$ converges to $\mu$ in 
\textbf{energy norm}. 
Convergence in energy norm
implies weak$^*$ convergence (see the second bullet point on \cite[p. 306]{Chi19}, or \cite[Theorem~7.3.10]{Hel14} for the case of potential theory in $\mathbb R^n$), i.e.,
\[ \lim_{n \to \infty} \int f d\mu_n = \int f d\mu \]
for every continuous $f$ on $K$. 

For a compact set $F \subset X \setminus \{p_{\infty}\}$ and a lower semi-continuous $\varphi : F \to \mathbb R \cup \{+\infty\}$, we denote
\begin{equation} \label{GreenEF}
	E_{\varphi}(F) = \inf_{\mu \in \mathcal E^1(F)}
	E_{\varphi}\left[\mu\right]. \end{equation}
\begin{definition} Let $F \subset X \setminus \{p_{\infty}\}$ be compact and let  $\varphi : F \to \mathbb R \cup \{+\infty\}$ $E_{\varphi}(F) < \infty$ be a lower semi-continuous function. 
	Suppose $E_\varphi(F) < +\infty$.
	Then a probability measure $\mu$ on $F$
	that satisfies $E_{\varphi}(F) = E_{\varphi}\left[\mu\right]$
	is called
	an \textbf{equilibrium measure} of $F$ in the external field $\varphi$, or simply an equilibrium measure. \end{definition}

There is only one equilibrium measure, so that we speak of the equilibrium measure of $F$ in the external field $\varphi$, as is part of the following proposition.

\begin{proposition} \label{prop:eqmeasure} Suppose $F \subset X \setminus \{p_\infty\}$
	is a compact set, and $\varphi : F \to \mathbb R \cup \{+\infty\}$ is lower semi-continuous with $E_{\varphi}(F) < +\infty$.
	Then there is a unique equilibrium measure $\mu^F$
	on $F$ in the external field $\varphi$. 
	The equilibrium measure is the
	unique probability measure $\mu$ on $F$ such that for
	some constant $c$, 
	\begin{align} \label{GreenEL1} 
		2 \int G(p,q) d\mu(q) + \varphi(p) & = c,
		\quad \text{q.e.\ on } \supp(\mu),  \\ \label{GreenEL2}
		2 \int G(p,q) d\mu(q) + \varphi(p) & \geq c,
		\quad \text{q.e.\ on } F.  
	\end{align}
	Here q.e.\ means quasi-everywhere, i.e., except for
	a set of zero capacity.
\end{proposition}
\begin{proof}
With the aid of Proposition~\ref{prop:Gposdef}, the proof
is similar to the proof for equilibrium measures
in logarithmic potential theory given in \cite[Theorem~I.1.3]{ST97},
see also \cite[Section~2.3]{Chi19} or \cite[Lemma~3.4.5]{Ski15}. 

For the variational conditions \eqref{GreenEL1}
and \eqref{GreenEL2} it is important that $G$ is symmetric,
see Proposition~\ref{prop:bipolarGreen} (d), since otherwise
$2 \int G(p,q)d\mu(q)$ in \eqref{GreenEL1} and \eqref{GreenEL2}
should be replaced by
$\int \left(G(p,q) + G(q,p)\right) d\mu(q)$.
\end{proof}

In the present work we will have that $F$ is a connected compact
set with more than one point (i.e., a continuum), or a
finite disjoint union of such sets, and in
such a situation there are no exceptional sets of zero capacity, and \eqref{GreenEL1} and \eqref{GreenEL2} hold everywhere
on their respective sets.

\subsection{$S$-property}

For a compactly supported measure $\mu$ with $\supp(\mu) \setminus X \setminus \{p_\infty\}$, note that $p \mapsto \int G(p,q) \mu(q)$ is a harmonic function on
$X \setminus (\{p_\infty\} \cup \supp(\mu))$. 
It is 
real-valued, but being harmonic it is locally the
real part of a holomorphic function that we denote by $g$
and we call it the $g$-function.
Depending on the situation, we may need to define
certain branch cuts in order to make $g$ single-valued. 
Its real part
\[ \Re g(p) = - \int G(p,q) \mu(q) \]
is always single-valued, however.

In case $F = \gamma$ is a contour and $\varphi = \Re V$ for a multi-valued, locally 
meromorphic function on $X$ with single-valued real part, the variational 
conditions \eqref{GreenEL1}--\eqref{GreenEL2} can also
be formulated as
\begin{equation} \label{GreenEL3}
	\begin{aligned} 
		\Re \left( g_+(p) + g_-(p) - V(p) \right) & = c, 
		\qquad p \in \supp(\mu), \\
		\Re \left( g_+(p) + g_-(p) - V(p) \right) & \leq c, 
		\qquad p \in \gamma.
\end{aligned} \end{equation}
The contour will be oriented, which defines the $+$ and $-$
sides; the notation $g_+$ and $g_-$ refers to limiting values of 
$g$ on the $+$ and $-$ sides of $\gamma$.

\begin{definition}\label{def:S-property} A contour $\gamma$ has the 
	\textbf{$S$-property in external field} $\Re V$
	if the $g$-function of its equilibrium measure $\mu$ in the external field $\Re V$ has the 
	following property:  not only the real part of $g_+ + g_- - V$ is constant on 
	$\supp(\mu)$, see \eqref{GreenEL3}, but also the imaginary part of $g_+ + g_- - V$ 
	is piecewise constant on each connected component of $\supp(\mu)$.
\end{definition}
In case $\gamma$ has more than one connected component,
then the imaginary part of the constant can be different 
on the different components. The imaginary part can also
be different due to a different choice of branches
of $g$ in case a branch cut intersects $\gamma$.

\section{Statement of results}\label{sec:statement-of-results}

\subsection{Cauchy kernel}

We would like to have the analogue of 
Theorem~\ref{theo:MFR} (a) which in particular states that for 
a critical measure 
\begin{equation} \label{Cauchykerneleq} \left[ \int C(p,q) d\mu(q) - \frac{dV(p)}{2} \right]^2 \end{equation}
is a meromorphic quadratic differential on $X$ where $C(p,q)$
is an appropriate Cauchy kernel.

\begin{definition} \label{def:Cauchy10}
	The \textbf{Cauchy kernel} $C(p,q)$ is given in terms
	of the bipolar Green's function with pole at $p_{\infty}$, 
	see Definition~\ref{def:bipolarG}, by
	\begin{equation} \label{Cauchy10} 
		C(p,q) = -2 \partial_p G(p,q) dp. \end{equation}
\end{definition}

Here $2\partial_p = \partial_x - i \partial_y$ if $p = x+iy$
in a local coordinate. Since $p \mapsto G(p,q)$ is harmonic
for $p \in X \setminus \{q,p_{\infty}\}$, we find from
\eqref{Cauchy10} that $C(p,q)$ is a 
meromorphic differential in the $p$-variable (for any fixed $q \in X \setminus \{p_{\infty}\}$)
and a harmonic function in the $q$-variable. More precisely,
\begin{itemize}
	\item $C(\cdot, q)$ is a meromorphic differential on $X$,
	with simple poles at $q$ and at $p_{\infty}$, and
	holomorphic otherwise, the pole at $q$ has residue $1$
	and the pole at $p_\infty$ has residue $-1$ (these
	residues come from the behavior 
	\eqref{Green2}, \eqref{Green3} of the bipolar Green's function)
	
	\item $C(p,\cdot)$ is a harmonic function 
	on $X \setminus \{p, p_{\infty}\}$. 
\end{itemize}
We also note that, since $G(p,q)$ is harmonic and single-valued,
\begin{itemize}	
	\item $C(\cdot,q)$ has purely imaginary periods.
\end{itemize}
since $G(p,q)$ is harmonic and single-valued. The above
three properties actually characterize the Cauchy kernel.

For any measure $\mu$ with compact support, the expression
\eqref{Cauchykerneleq} is a quadratic differential that is meromorphic
on $X \setminus \supp(\mu)$ with poles at $p_{\infty}$
and at the poles of $dV$.

\begin{example} \label{example42}
	(a) On the Riemann sphere we have
	\[ C(p,q) = -2 \partial_p \left(\log \frac{1}{|p-q|} \right) dp
	= \frac{dp}{p-q}, \]
	which is the usual Cauchy kernel from complex analysis.
	
	(b) On a complex torus $X = \mathbb C \slash \Lambda$
	as in Example~\ref{example23} (b) we have
	\begin{align} \nonumber C(p,q) & = -2 \partial_p \left(
		\log \left| \frac{\theta_1(p) \theta_1(q)}{\theta_1(p-q)} \right|
		- \frac{2\pi}{\Im \tau} \left(\Im p\right) \left(\Im q\right)\right) dp \\ \label{Cauchy1} 
		& = - \left( \frac{\theta_1'(p)}{\theta_1(p)} - \frac{\theta_1'(p-q)}{\theta_1(p-q)} + \frac{2\pi i}{\Im \tau}
		\Im q\right) dp. \end{align}
	The formula shows that $q \mapsto C(p,q)$ is not  meromorphic 
	in the genus one case, but it is harmonic.
\end{example}

The notion of Cauchy kernel we use here could be further qualified as {\it imaginary normalized} and it is only harmonic with respect to $q\neq p$. More commonly the term is used to refer to a kernel which is {\it meromorphic} in $q$, with an additional pole at a prescribed collection of $g$ points (where $g$ is the genus of the Riemann surface). Possibly the first occurrence is in \cite{BS49} but the notion appears ubiquitously in the literature on Riemann surfaces, notably in \cite{Fay73, GuRo62, Zve71}. A related notion is used also later in this paper, see Section~\ref{sec:main-result}. 

\subsection{Critical measures}
\label{sec:critical-measures}

In order to define  the notion of a critical measure, we need an analogue of the Schiffer variation formula \eqref{muepsh}. 
Instead of considering a function $h$, as we did in the genus
zero case, we identify  $h$ as a vector field on $X$, i.e., a section 
of the tangent bundle in the higher genus case. 
A vector field $h$ induces a flow $\Phi(t,p)$ on $X$ (for $t \in \mathbb R$ and
$p \in X$, we have $\Phi(t,p) \in X$) satisfying
$\frac{d\Phi}{dt} = h$ and $\Phi(0,p) = p$. 
Then, given a measure $\mu$ on $X$ we define $\mu_{\varepsilon,h}$ by its action
on continuous functions $f$
\begin{equation} \label{CMhigher1} 
	\int f d\mu_{\varepsilon,h} = \int f(\Phi(\varepsilon,p)) d\mu(p), \qquad \varepsilon \in \mathbb R, \end{equation}
which is the analogue of \eqref{muepsh}. 
Thus $\mu_{\varepsilon,h}$ is the image of $\mu$ along the flow
induced by $h$, and we may alternatively write
\[ \mu_{\varepsilon,h} = \Phi(\varepsilon, \mu). \]

\begin{definition} \label{def:CMhigher}
	Let $\mu$ be a measure on $X$ with support in $X \setminus \{p_\infty\}$ and suppose that the external field $\varphi$ is a real-valued $C^1$ function on a neighborhood of $\supp(\mu)$.
	Then $\mu$ is a \textbf{critical measure} in the external field $\varphi$ if
	\begin{equation} \label{CMhigher2} \lim_{\varepsilon \to 0} \frac{E_\varphi[\mu_{\varepsilon,h}] - E_\varphi[\mu]}{\varepsilon} = 0 
	\end{equation}
	for every $C^1$ vector field $h$. 
\end{definition}

Similar to \cite[Lemma~3.1]{MFR11}, there is a convenient identity for the limit in \eqref{CMhigher2} in case $\varphi$ is given by the real part of $V$.

\begin{proposition}\label{prop:CMhigher}
	Suppose that the external field $\varphi$ is given by $\varphi = \Re V$ and suppose that $\mu$ is a measure with support in $X \setminus \{p_\infty\}$ 
	such that $E[\mu] < \infty$ and $\varphi$ is bounded on $\supp(\mu)$.
	Then the following holds.
	\begin{enumerate}
		\item[\rm (a)] 	
		For every $C^1$ vector field $h$, the limit in \eqref{CMhigher2} exists with
		\begin{equation} \label{CMhigher4} \lim_{\varepsilon \to 0 } \frac{E_\varphi[\mu_{\varepsilon,h}]-E_\varphi[\mu]}{\varepsilon} 
			=  \Re D_{V,h}(\mu),  \end{equation}
		where $D_{V,h}(\mu)$ is given by 
		\begin{equation} \label{CMhigher3}
					D_{V,h}(\mu) = - \iint 
					\left(h(p) C(p,q) + h(q) C(q,p)\right) d\mu(p) d\mu(q) 
					+  \int h dV d\mu \end{equation}
				and $C(p,q)$ is the Cauchy kernel from Definition~\ref{def:Cauchy10}, see Remark \ref{remarkDVh}.
		\item[\rm (b)] 
		A critical measure $\mu$ satisfies
		$D_{V,h}(\mu) = 0$
		for every $C^1$ vector field $h$.
	\end{enumerate}	
\end{proposition}	

The proof is given in Section~\ref{sec:proofs-critical}.

\begin{remark} \label{remarkDVh}
Let us comment on the expression \eqref{CMhigher3}.
The vector field $h$ acts on the meromorphic differential $dV$ in a natural way: the expression $h dV$ defines a $C^1$ function on $X$ away from the poles of $dV$.
The function can  be integrated against a measure and this is how
the term $\int  h dV  d\mu$
should be interpreted in \eqref{CMhigher3}. It is
the analogue of the right-hand side of \eqref{CM1}. 
Also, the fact that the external field is harmonic means that we only need to consider vector fields in the \textit{holomorphic tangent bundle}.
See  Remark~\ref{remT10} for more details.

The analogue of the left-hand side of \eqref{CM1} is the double integral 
with the Cauchy kernel $C(p,q)$. 
Since $C(p,q)$ is a differential in $p$, the product
$h(p) C(p,q)$ is a function in both variables $p$ and $q$
that becomes infinite when $p=q$. The same holds true
for $h(q) C(q,p)$, but the infinities for $p=q$ disappear 
in the sum
$h(p) C(p,q) + h(q) C(q,p)$ if $h$ is
a $C^1$ vector field. The sum is a well-defined bounded and continuous function on 
$(X \setminus \{p_{\infty}\}) \times (X \setminus \{p_{\infty}\})$ that 
plays the role of the divided difference $\frac{h(s)-h(t)}{s-t}$ in \eqref{CM1}.
\end{remark}

\begin{remark}
	\label{remT10}
	The energy functionals we are considering involve harmonic functions $\varphi = \Re V$ for their external field, and the bipolar Green's function is also a harmonic function.
	This has the following effect: when we consider the Schiffer variations, it is sufficient to consider vector fields in the  holomorphic tangent bundle of the Riemann surface $X$, usually denoted by $T^{1,0} = T^{1,0}X$. 
	In concrete terms, this means that when we consider $X$ as a two-dimensional real manifold with local coordinates $x,y$ (forming the real and imaginary part of a local coordinate $z=x+i y$), a vector field is simply an expression of the form $\mathbb V=a(x,y)\frac {\partial}{\partial x} +  b(x,y)\frac {\partial}{\partial y} $, with $a,b$ smooth real-valued functions. The vector field can then be written in terms of the Cauchy-Riemann operators $\frac {\partial}{\partial z} = \frac 12 \left(\frac \partial{\partial x} -i \frac \partial{\partial y}\right), \frac {\partial}{\partial \overline z} = \frac 12 \left(\frac \partial{\partial x} +i \frac \partial{\partial y}\right)$ as 
	\begin{equation}
		\mathbb V = (a + i b) \frac {\partial}{\partial z}  +  
		(a - i b) \frac {\partial}{\partial \overline z}  = \mathbb V^{1,0} + \mathbb V^{0,1}.
	\end{equation}
	A real-valued harmonic function $H(x,y)$ can be written locally as the real part of a holomorphic (or anti-holomorphic) function $f$ as $H(x,y) =  \Re f(z)$. Acting with $\mathbb V$ on such a  function yields 
	\begin{equation}
		\mathbb  V H(x,y) =\frac{1}{2}
		\left( (a + i b) \frac {\partial f(z)}{\partial z} + (a - i b) \frac {\partial \overline {f(z)} }{\partial \overline z}\right) = 
		\Re \left( \mathbb V^{1,0} f \right).
	\end{equation}
	Having thus established that we need only to consider vectors in $ T^{1,0}$, we also now point out that given any (meromorphic) differential, written in local coordinate as $\omega = f(z) d z$, the expression $\frac 1{\omega}$  can be interpreted as a (meromorphic) vector field in $T^{1,0}$.  This is easily verified by observing that it transforms as a vector in $T^{1,0}$ does under holomorphic change of coordinates. 
	In fact this also holds for {\it smooth} sections of the canonical bundle, namely, expressions $\omega  = f(z,\overline z)d z$, with $f$ a smooth (i.e. not necessarily satisfying the Cauchy-Riemann equations) complex-valued function. 
	Even in this case, $\frac1\omega$ again defines a smooth vector field wherever $\omega$ is not zero. 
\end{remark}

\subsection{Critical sets} \label{sec:critical-sets}

For a vector field $h$ with an associated flow $\Phi$, we 
also consider the flow   
\begin{equation} \label{CShigher1} F_{\varepsilon,h} = \Phi(\varepsilon, F) \end{equation}
of a compact set $F$.
\begin{definition} \label{def:CShigher}
	$F$ is a \textbf{critical set} in the external field $\varphi = \Re V$ 	if
	\begin{equation} \label{CShigher2} \lim_{\varepsilon \to 0} \frac{E_{\varphi}(F_{\varepsilon,h}) - E_{\varphi}(F)}{\varepsilon} = 0 \end{equation}
	for every $C^1$ vector field $h$. 
\end{definition}

The following analogue of Proposition~\ref{prop:CS} holds.
\begin{proposition} \label{prop:CShigher}
	Let $F \subset X \setminus \{p_{\infty}\}$ be a union of continua such that $\Re V$ is
	bounded from below on $F$. Let $\mu^F$ be the equilibrium measure of $F$ in the external field $\varphi = \Re V$. 
	\begin{enumerate}
		\item[\rm (a)] 
		Then for every $C^1$ vector field $h$,
		\begin{equation} \label{CShigher3} \lim_{\varepsilon \to 0} \frac{E_{\varphi}(F_{\varepsilon,h}) - E_{\varphi}(F)}{\varepsilon} = 
			\lim_{\varepsilon \to 0} \frac{E_{\varphi}[\mu^F_{\varepsilon,h}] - E_{\varphi}[\mu^F]}{\varepsilon}		
		\end{equation}
		\item[\rm (b)] $F$ is a critical set in the external field $\varphi$ if and only
		its equilibrium measure $\mu^F$ is a critical measure in the external field $\varphi$.
	\end{enumerate}
\end{proposition}
The proof is given in Section~\ref{sec:proofs-critical},
and it is similar to the proof of Lemma~3.5 in \cite{Rak12} and the proof of Proposition~4.2 in \cite{KS15}. 
Proposition~\ref{prop:CShigher} is the higher genus analogue 
of parts (a) and (b) of Proposition~\ref{prop:CS}.
The analogue of part (c) is now immediate as well.
\begin{corollary} \label{cor:maxmin}
	Suppose $F$ belongs to a family $\mathcal F$
	of union of continua such that 
	for every $C^1$ vector field $h$ there is $\varepsilon_0 >0$	such that
	$F_{\varepsilon,h} \in \mathcal F$ for every $\varepsilon \in
	(-\varepsilon_0, \varepsilon_0)$, where $F_{\varepsilon,h}$ is given by \eqref{CShigher1}. Suppose also that  
	\begin{equation} \label{CShigher12}
		E_\varphi(F) = \max_{F'\in \mathcal F} E_\varphi(F'). 
	\end{equation}
	Then $F$ is a critical set and $\mu^F$ is a critical measure
	in the external field $\varphi$.
\end{corollary}

\subsection{Max-min energy problem}

Next we describe a situation to which  Corollary \ref{cor:maxmin}
applies and where we can prove the existence of a critical set $F$ and
a critical measure $\mu^F$ in external field. 
We are not aiming at the most general set-up here,
but instead we confine ourselves to a case that will be 
useful for the analysis of periodic tiling models, see Section~\ref{sec:motivation}.

We assume $\varphi = \Re V$, where $V$ is locally meromorphic on $X$
with logarithmic singularities at a finite number of points. 
The differential $dV$ is meromorphic and single-valued on $X$
with at most simple poles (i.e., an Abelian differential of the third kind). 
The logarithmic singularities of $V$ correspond to simple poles of $dV$. 
We assume the residues are real and negative, 
 except for a pole $p_0 \in X \setminus \{p_{\infty} \}$
with positive residue $r_0 > 0$ and a possible pole at $p_\infty$.
This implies
that in a local coordinate $z_0$ at $p_0$ one has
\begin{equation} \label{phiatp0}
	\Re V(p) = r_0 \log |z_0(p)| + O(1), \quad \text{ as } p \to p_0.
\end{equation}
We fix a local coordinate at $p_0$, that identifies a certain
neighborhood of $p_0$ in $X$ with a disk 
$D(0,\delta_0) = \{|z_0| < \delta_0\}$  for some $\delta_0 > 0$.
For $\delta < \delta_0$ we use $U_{\delta}(p_0)$ to denote
the  neighborhood of $p_0$ that corresponds to 
$D(0,\delta)$ in this local coordinate.

The differential $dV$ could have simple pole at $p_{\infty}$ as well,
and then we use $r_{\infty} = \Res(dV,p_{\infty})$ 
to denote the residue which could be positive or negative.  
If $dV$ is holomorphic at $p_{\infty}$,
then we write $r_{\infty} = 0$. In either case we have
in a fixed local coordinate $z_{\infty}$ at $p_{\infty}$   
\begin{equation} \label{phiatpinfty}
	\Re V(p) = r_{\infty} \log |z_{\infty}(p)| + O(1),
	\quad \text{ as } p \to p_{\infty}.
\end{equation}
A neighborhood of $p_{\infty}$ is identified with the disk 
$D(0, \delta_{\infty})$ for some $\delta_{\infty} > 0$. 
For $\delta < \delta_{\infty}$ we use 
$U_{\delta}(p_{\infty})$ for the neighborhood of $p_{\infty}$
that corresponds to the smaller disk $D(0,\delta)$.

At all other poles (if any) there is a negative residue,
which means that $\Re V$ tends to $+\infty$ at those poles.
Thus the external field $\Re V$ is bounded from below on 
$X \setminus (U_{\delta}(p_0) \cup U_{\delta}(p_{\infty}))$
for any small enough $\delta > 0$. 
See Figure \ref{figcontours} for a sketch of the neighborhoods
$U_{\delta}(p_0)$ and $U_{\delta}(p_{\infty})$.
\begin{figure}
\resizebox{0.99\textwidth}{!}{
\begin{tikzpicture}[scale=7]
\begin{scope}[rotate=10]

\shade  [line width = 1.5, top color = white!76!black](0.5,-0.3)  
  to  [out =130, in =-30] (0.5,0.3) 
  to [out = 220, in=0] (0,0.1)
  to [out=180, in=-30] (-0.5,0.3) 
  to [out=150,  in =20] (-1.2,0.4)  
  to [out=200, in =-10] (-2,0.2)  
  to [out=210, in=30]  (-2,-0.2) 
  to [out=30, in=160, looseness=0.9] (-1.2,-0.4)  
  to [out=-20, in=210, looseness=0.9] (-0.5,-0.3)
  to [out=30, in =180] (0,-0.1) 
  to [out=0, in=150]  (0.5, -0.3);

\draw [line width = 1.5](0,0.1) to  [out=0,in=220] coordinate [  pos=0.4] (d2)(0.5,0.3)  coordinate (d) 
  to [out = -30, in =130] coordinate [pos=0.5] (cc) (0.5,-0.3) to [out=70, in=-80] (cc);
\shade [fill=white!50!black!, opacity =0.4]  (cc) to [out=-110, in=130] (0.5,-0.3) to [out=70, in=-80] (cc) ; 
\draw (d) to [out = -30, in =130] coordinate [pos=0.5] (cc) (0.5,-0.3) to [out=70, in=-80] (cc);

\draw [line width = 1.5] (0,-0.1) to [out=0, in = 150, looseness=01]  coordinate [ pos=0.4] (d2b)(0.5,-0.3);
 \draw [line width = 1.5](0,-0.1) to  [out=180, in = 30, looseness=01] coordinate [ pos=0.4] (c2b) (-0.5,-0.3);

\draw [line width = 1.5](0.5,-0.3)  to  [out =130, in =-30] (0.5,0.3) to [out = 220, in=0] (0,0.1)
  to [out=180, in=-30] coordinate [pos=0.4] (c2) (-0.5,0.3) coordinate(c)
  to [out=150,  in =20, looseness=1] (-1.2,0.4) coordinate (f) 
  to [out=200, in =-10, looseness=1] coordinate[pos=0.5] (t1) (-2,0.2) coordinate(f1)  
  to [out=210, in=30]  coordinate [pos=0.5] (f2) coordinate [pos=0.8] (f22)(-2,-0.2) coordinate (f3)
  to [out=30, in=160, looseness=0.9] coordinate [pos=0.5] (t3)  (-1.2,-0.4)  
  to [out=-20, in=210, looseness=0.9] coordinate (bottom) (-0.5,-0.3)
  to [out=30, in =180] (0,-0.1) 
  to [out=0, in=150]  (0.5, -0.3);

\draw [ green!30!black, line width = 1.5  , postaction={decorate,decoration={markings,mark=at position 0.6 with {\arrow[line width=1.8pt]{>}}}}]  (t1) to [out=-120, in=120] (t3);
\draw [ green!30!black, dashed, line width = 1.5 ]  (t1) to [out=-70, in=70] (t3);

\node at  ($(t1)+(-0.1,-0.3)$) {$\gamma_1$};
 
 \draw [ green!30!black, line width = 1.5  , postaction={decorate,decoration={markings,mark=at position 0.6 with {\arrow[line width=1.8pt]{<}}}}]  (c2) to [out=-60, in=60] (c2b);
 \draw [green!30!black,  line width = 1.5  , dashed]  (c2b) to [out=120, in=-120] (c2);
 
 \node at  ($(c2)+(0.07,-0.1)$) {$\gamma_2$};

\draw [line width =1.8, line cap =round] (-1.3,0.05) to [out=-20, in=200] coordinate[pos=0.1] (c1)  coordinate [pos=0.88] (c2) (-0.7,0.05)  ; 
\draw [line width = 1.4] (c1) to [out=40,in=145, looseness=1] (c2);

\draw [fill=white] (c1) to [out=40, in = 145] (c2) to [out =197, in =-17] (c1);
\draw [line width =1.8, line cap =round] (-1.3,0.05) to [out=-20, in=200]  (-0.7,0.05)  ;

\shade[ draw,fill=red!80!black!, opacity =0.4] (f3) to [out=110, in=-100] (f2) to [out=-70, in=30] (f3) -- cycle; 
\draw[ line width = 1pt] (f3) to [out=110, in=-100] (f2) to [out=-70, in=30] (f3) -- cycle; 

\coordinate (v1) at  (-1.2,-0.2) ;
\coordinate (v2) at  ($  (v1) + (60:0.1)$) ;
\coordinate (v3) at  ($  (v1) + (-80:0.1)$)  ;

\coordinate  (v4) at (-1.1, 0.3);
\coordinate (p1) at (-0.1,0);
\coordinate (p2) at (0.1,0.05);
\coordinate (t1) at (-0.5,0.1);

\draw [ line width =0.2pt, fill=white!1!black!80!red, fill opacity = 0.3 , rotate=-10]   (v1)++(0.02,0) ellipse (0.05 and 0.05);
\node at ($(v1)+(0,-0.09)$) {$U_\delta(p_0)$};

\draw [ line width =0.2pt, fill=white!1!black!80!red, fill opacity = 0.3 , rotate=-10]   (v4)++(0.02,0) ellipse (0.05 and 0.05);
\node at ($(v4)+(0,-0.08)$) {$U_\delta(p_\infty)$};
\node at ($(v4)+(0.15,0)$) {$\gamma_3$};

\draw [green!30!black, line width = 1.2  , postaction={decorate,decoration={markings,mark=at position 0.6 with {\arrow[line width=1.6pt]{>}}}}] (v4)+(-0.01,-0.02) circle [radius =0.12];
\end{scope}
\end{tikzpicture}
}
\caption{
Neighborhoods $U_{\delta}(p_0)$ and $U_{\delta}(p_{\infty})$,
and an example  of a multi-contour $\gamma = \gamma_1 \cup \gamma_2 \cup \gamma_3$ in $\mathcal T_\delta$ as defined in Definition~\ref{def:Tdelta}.
The three components belong to $\mathcal S_\delta$.
\label{figcontours}}
\end{figure}
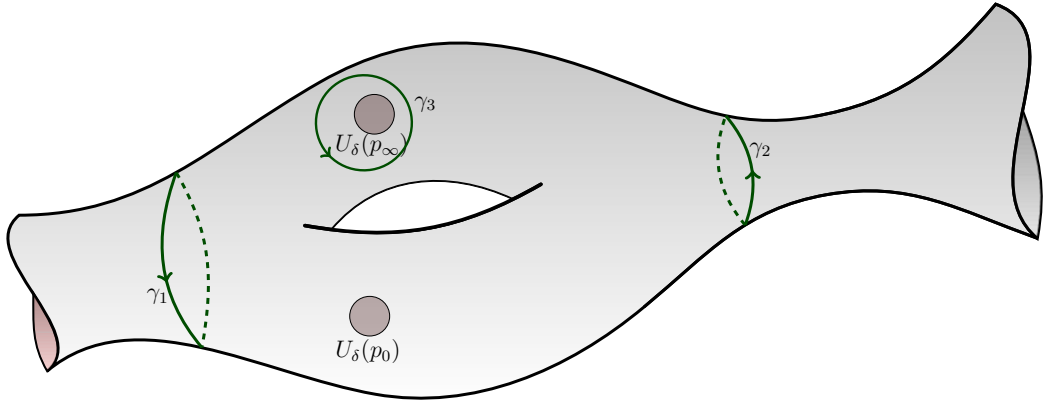

\begin{definition}\label{def:Tdelta}
	Suppose $\delta_0$ and $\delta_{\infty}$ are 
	as above with $\delta_{\infty} \geq \delta_0$ and such that the neighborhoods $U_\delta(p_0)$ and $U_\delta(p_{\infty})$ are disjoint
	for every $\delta \in (0,\delta_0]$. 
	Then for $\delta \in (0,\delta_0]$, we define the following
	\begin{enumerate}
	\item[\rm (a)] 
	We use $\mathcal S_{\delta}$ to denote the set
	of simple closed oriented contours in $X \setminus (U_{\delta}(p_0) \cup U_{\delta}(p_{\infty}))$ 
	that are not homotopic to a point in $X \setminus \{ p_0,p_\infty\}$. 
		\item[\rm (b)]  
	We use $\mathcal T_{\delta}$ to denote the set
	\begin{equation} \label{eq:defTdelta} \mathcal T_{\delta} = \left\{ \gamma = \gamma_1 \cup \cdots \cup \gamma_n
	\mid  
	\begin{array}{ll} n \in \mathbb N,  \gamma_j \in \mathcal S_{\delta} \text{ for every $j$, and }  \gamma \text{ is homologous} \\ \text{in } X \setminus \{p_0, p_{\infty}\}
		\text{ to a circle around $p_0$
				in } U_{\delta}(p_0) 
			 \end{array} \right\}. \end{equation}
	\item[\rm (c)]
	We define
	\begin{equation} \label{Hclosure}
		\mathcal F_{\delta}  = \overline{\mathcal T}_{\delta}, 
		\qquad 
		\mathcal F  = \bigcup_{0 < \delta < \delta_0} 
		\mathcal F_{\delta},
	\end{equation}
	where the closure is with respect to the topology corresponding
	to Hausdorff distance on compact subsets of $X$.
	\end{enumerate}
\end{definition}
An illustration of a possible multi-contour in $\mathcal T_\delta$ is provided in Figure~\ref{figcontours}.

The motivation for the definition~\eqref{eq:defTdelta}
comes from the study of orthogonality on the Riemann surface, see the discussion in Section~\ref{section45} below. The orthogonality 
is expressed through the vanishing of certain
integrals of meromorphic differentials over a small contour around $p_0$. The meromorphic differentials have poles at $p_0$ and $p_{\infty}$, and in the absence of further poles, the integral 
can be deformed to a system of contours $\gamma$ as 
in \eqref{eq:defTdelta} by Cauchy's theorem.   
Within the class $\mathcal T_{\delta}$
	(with $\delta$ sufficiently small) 	we want to find an ideal $\gamma$.  

The Hausdorff distance on compact subsets of $X$ that is used
in \eqref{Hclosure} comes from 
a distance function $d$ on $X$ that is compatible with the complex structure; see \eqref{eq:defRiemDistance}. Then  we write
\[ d(p,F) = \min_{q \in F} d(p,q) \] 
for the distance of $p \in X$ to a compact set $F \subset X$, 
and 
\[ d_H(F_1, F_2) = \max\left( \max_{p \in F_1} d(p,F_2),
\max_{p \in F_2} d(p,F_1) \right) \]
for the Hausdorff distance between compact sets $F_1$ and $F_2$.

The Hausdorff distance restricted
to compact subsets of a fixed compact defines a compact
metric space. Thus $\mathcal F_{\delta}$ is compact for
every $\delta \in (0, \delta_0)$. The set $\mathcal F$ from \eqref{Hclosure}
does not depend on the precise choice of distance
function on $X$. 
Recall that $E_{\varphi}(F)$ is given by \eqref{GreenEF}.

\begin{theorem} \label{thm:residues}
	In the above setting, suppose that 
	$r_0 = \Res(dV,p_0) > 1$ and $r_{\infty} = \Res(dV,p_{\infty}) > -1$.
	Then there exists $F \in \mathcal F$
	(as defined in Definition~\ref{def:Tdelta})  such that 
	\begin{equation} \label{eq:residues} 
		E_{\varphi}(F) = \max_{F'\in \mathcal F} E_{\varphi}(F'). 
		\end{equation}
	The set $F$ is critical and its equilibrium measure $\mu^F$
	is a critical measure in the external field $\varphi$.
\end{theorem}

The condition $r_0 > 1$ in Theorem~\ref{thm:residues} implies that any $F \in \mathcal F$ that
comes very close to $p_0$ has a small bipolar Green's energy $E_{\varphi}(F)$ in a sense that will be made precise in
Proposition~\ref{prop:residues}.
The condition $r_{\infty} > -1$ gives the same result for $F$ that
come close to $p_\infty$.   

The proof of Theorem~\ref{thm:residues} is in Section~\ref{sec:maxmin}.

\subsection{Final result}\label{sec:main-result}

The final result of the paper is a generalization of 
Theorem~\ref{theo:MFR} to the case of a higher genus
Riemann surface, albeit only for a special case of genus one
that is given by a cubic equation
\begin{equation} \label{Xgenus1} 
	X: \quad w^2 = z(z-x_1)(z-x_2) \end{equation}
with $x_1, x_2 \in \mathbb R$ and $x_1 < x_2 < 0$. 
We use $p = (w,z)$ to denote a generic point on $X$. We use $p_{\infty}$
for the point at infinity and we use $p_0 = (0,0)$, 
$p_1 = (0,x_1)$, $p_2 = (0,x_2)$ to denote the branch
points of \eqref{Xgenus1}.  
Then $X$ has two sheets that are copies of $\mathbb C \setminus
( (-\infty,x_1] \cup [x_2,0])$. The first sheet is such
that $w > 0$ for $z > 0$ on the first sheet, while
$w < 0$ for $z > 0$ on the second sheet. $X$ has an
antiholomorphic involution 
\begin{equation} \label{involution} 
	\sigma : X \to X : (w,z) \mapsto (\overline{w},\overline{z}). \end{equation}
The real locus of $X$ is invariant under $\sigma$, and 
it consists of two parts. The bounded part is the real
oval $C_1$ consisting of the points $p = (w,z)$ with $z \in [x_1,x_2]$, and the unbounded part $C_2$ consists of the
points $p = (w,z)$ with $z \in [0,\infty)$ together with 
the point $p_{\infty}$ at infinity, see Figure~\ref{figovals}.

By means of the Abel map, we identify $X$ with a complex
torus  $\mathbb C \slash (\mathbb Z + \tau \mathbb Z)$ with $\tau \in i \mathbb R^+$, with $p_\infty$  identified with $0$. 
The complex torus has the following $(2,-1)$-Cauchy kernel,
which depends on a parameter $a$. 
Here $(2,-1)$ refers to the fact that this kernel is a quadratic
differential with respect to one variable and a meromorphic vector field
with respect to the other variable.

\begin{proposition} \label{prop:C21kernel}
	For $u,v,a \in  \mathbb C \slash (\mathbb Z + \tau \mathbb Z)$
	we have that 
	\begin{equation} \label{Cauchy21}
		C^{(2,-1)}(u,v;a)  
		= \frac{\theta_1'(0) \theta_1(u-a)^2 \theta_1(v)^2 \theta_1(u-v+2a)}{\theta_1(u-v)\theta_1(v-a)^2 \theta_1(u)^2 \theta_1(2a)}
		\frac{du^2}{dv} \end{equation} 
	satisfies the following.
	\begin{enumerate}
		\item[\rm (a)] For fixed $u,a$, we have that \eqref{Cauchy21}
		is a meromorphic vector field (i.e.\ a $-1$-differential by Remark~\ref{remT10}) in $v$ 
		with a simple pole at $v=u$, a double pole at $v=a$, a double
		zero at $v=0$ and a simple zero at $v = u + 2a$.
		\item[\rm (b)] For fixed $v, a$, we have that  
		\eqref{Cauchy21} is a meromorphic quadratic differential in $u$, with a simple pole
		at $u=v$, a double pole at $u=0$, a double zero at $u=a$, 
		and a simple zero at $u = v-2a$.
		\item[\rm (c)] The kernel \eqref{Cauchy21} is normalized such that
		\begin{equation} \label{CMQD5} 
			C^{(2,-1)}(u,v;a) \frac{dv}{du^2} = \frac{1}{u-v} + O(1) \end{equation}
		as $u \to v$ on the complex torus.
	\end{enumerate}
\end{proposition}
\begin{proof}
	The periodicity properties \eqref{Jacobi1periods} show that
	\eqref{Cauchy21} is doubly periodic in both variables $u$ and $v$.
	The statements (a), (b), (c) can be directly verified from the formula \eqref{Cauchy21} with the fact that $\theta_1$ has a simple zero at $0$
	and no other zeros (modulo the lattice). 
\end{proof}

\begin{figure}
\begin{center}
\begin{tikzpicture}[scale=0.9]
	\draw[help lines, color=black!10] (-2,-3) grid (1,3);
  \draw[->] (-2, 0) -- (1, 0) node[right] {$z$};
  \draw[->] (0, -3) -- (0, 3) node[above] {$w$};
\node at (-1.5,0.9) {$C_1$};
\node at (0.7,0.9) {$C_2$};
\draw[ domain=-1.8:-1, smooth, samples=1000, variable=\x, blue, thick] plot  (\x,{sqrt(\x*(\x+1)*(\x+1.8))});
\draw[ domain=-1.8:-1, smooth, samples=1000,variable=\x, blue, thick] plot  (\x,{-sqrt(\x*(\x+1)*(\x+1.8))});
\draw[ domain=0:1, smooth, samples=1000,variable=\x, blue, thick] plot  (\x,{-sqrt(\x*(\x+1)*(\x+1.8))});
\draw[ domain=0:1, smooth, samples=1000,variable=\x, blue, thick] plot  (\x,{sqrt(\x*(\x+1)*(\x+1.8))});
\end{tikzpicture} 
\qquad\qquad
\begin{tikzpicture}[scale=4]
\coordinate (tau) at (0,1.2);
\draw [fill=black!10!white] (0,0) to (tau)  to ($(tau)+ (1,0)$) to (1,0) to cycle;

\draw [blue,line width=1, postaction={decorate,decoration={{markings,mark=at position 0.7 with {\arrow[line width=1.5pt]{>}}}} }]($0.5*(tau)$) to node[pos=0,left] {$\frac \tau 2$} node[pos=0.6,below]{$C_1$}($0.5*(tau)+(1,0)$);

\draw [blue,line width=1, postaction={decorate,decoration={{markings,mark=at position 0.7 with {\arrow[line width=1.5pt]{>}}}} }](0,0) to node[pos=0.6,below]{$C_2$} (1,0);

\node [below]at (0,0) {$0$};
\node [above]at (tau) {$\tau$};
\node [above]at ($(tau)+(1,0)$) {$\tau+1$};
\node [below]at (1,0) {$1$};
\end{tikzpicture}
\end{center}
\caption{
Left: depiction of the real locus of an elliptic curve  of the form \eqref{Xgenus1} where $C_1$ is the bounded real oval and $C_2$ is  unbounded. Right: the two real ovals in the complex torus. For a curve of the type \eqref{Xgenus1} the parameter $\tau$ is purely imaginary, $\tau \in i\mathbb R^+$.  
\label{figovals}}
\end{figure}
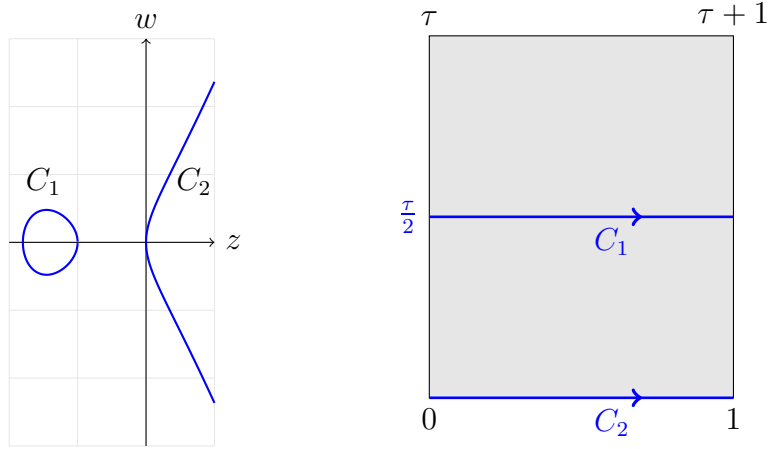

Because of \eqref{CMQD5} we will say that $C^{(2,-1)}(u,v;a)$
has residue $1$ at $u=v$. This notion of residue is conformally
invariant if we transform $u$ and $v$ simultaneously 
by the same conformal map.

As a consistency check, we note that the poles in part (a)
of Proposition~\ref{prop:C21kernel} add up to $u+2a$, and this is
where the simple zero is. The zeros and poles in part (b) both
add up to $v$, as it should be in view of Abel's theorem. It is in
accordance with the fact that 
$C^{(2,-1)}(u,v;a) \frac{dv}{du^2}$ is doubly periodic with periods 
$1$ and $\tau$ in both $u$ and $v$.

With the Abel map we transform  \eqref{Cauchy21} to the Riemann
surface \eqref{Xgenus1} where ${p_{\infty} \in X}$ corresponds
to $0$ on the complex torus.  We still use $C^{(2,-1)}$ to
denote the $(2,-1)$-Cauchy kernel on $X$. 
There is actually an explicit expression in terms of the representation \eqref{Xgenus1}, thus side-stepping the necessity of using $\theta$--functions. Taking $p = (w,z)$,  
	$q = (\wt w, \wt z)$ and $a = (w_0,z_0)$, one may verify that 
	\begin{multline}
		C^{(2,-1)}(p,q;a) \\ =
		\left(\frac { w+\wt w}{z-\wt z} - \frac {w_0+ \wt w}{z_0-\wt z} +\left( \frac {w_0'}{z_0 - \wt z} - \frac {w_0+ \wt w}{(z_0-\wt z)^2} \right) (z-z_0) \right) \frac {d z^2}{2w^2} \frac{ \wt w}{d \wt z},
	\end{multline}
	where $w_0'$ is the derivative $\frac {d}{d z}\sqrt{z(z-x_1)(z-x_2)}\bigg|_{z=z_0}$.

By Proposition~\ref{prop:CMhigher} a critical measure $\mu$ satisfies
\begin{equation} \label{CMQD1} 
	\iint \left(h(p) C(p,q) + h(q) C(q,p)\right) d\mu(p) d\mu(q) 
	= \int h dV \, d\mu \end{equation} 
for every $C^1$ vector field $h$. 
Choosing $h = C^{(2,-1)}(p, \cdot)$ allows us to extract relevant
information from \eqref{CMQD1}.

Now we arrive at a generalization of Theorem~\ref{theo:MFR} 
to a symmetric genus one case.

\begin{theorem} \label{theo:QDhigher}
	Let $X$ be given by \eqref{Xgenus1}.
	Suppose that the external field $\varphi$ is given by $\varphi = \Re V$ for a multi-valued, locally meromorphic function $V$ on $X$ with a single-valued real part and that $dV$ is a  meromorphic differential on $X$. 
	Suppose $\mu$ is a critical measure in the external
	field $\varphi = \Re V$ that is invariant under
	the involution \eqref{involution}. Suppose $\supp(\mu) \subset X \setminus (\{p_{\infty}\} \cup C_1)$ where $C_1$ is the real
	bounded oval of $X$.  
	Then there exists an $a \in C_1$ such that the following holds.
	\begin{enumerate}
		\item[\rm (a)] Then
		\begin{equation} \label{QDhigher1} 
			\left[ \int C(u,q) d\mu(q) - \frac{dV(u)}{2} \right]^2
			= \upomega(u), \qquad a.e.\ \text{ on } X, \end{equation}
		where $\upomega$ is the meromorphic quadratic differential
		\begin{equation} \label{QDhigher2} \upomega(u) = \left(\frac{dV(u)}{2} \right)^2
			- \int (C(u,q) dV(u) - C^{(2,-1)}(u,q;a) dV(q) ) 
			d\mu(q).  \end{equation}
		\item[\rm (b)] The support of $\mu$ is a union of analytic
		arcs or loops that are maximal trajectories of $-\upomega$.
		The measure $\mu$ is absolutely continuous with respect
		to arclength, and it is given by the differential 
		\[ d\mu(u) = \frac{1}{\pi i} \upomega(u)^{1/2},
		\qquad u \in \Sigma := \supp(\mu), \]
		with an appropriate branch of the square root on each
		open 	arc or closed loop.
		\item[\rm (c)] The bipolar Green's potential
		\[ G^\mu(p) := \int G(p,q) d\mu(q) \]
		of $\mu$ satisfies
		\[ 2 G^\mu(p) + \Re V(p) = c_j, \qquad p \in \Sigma_j \]
		with a constant $c_j$ that can be different on
		each component $\Sigma_j$ of $\Sigma$.
		\item[\rm (d)] Any point $p \in \Sigma$ that is
		not a zero of $\upomega$ has a neighborhood $D$ such
		$D \cap \Sigma$ is an analytic arc, and
		\begin{equation}\label{eq:S-property}\frac{\partial}{\partial n_+}
		\left(2 G^\mu + \Re V \right)
		= \  \frac{\partial}{\partial n_+}
		\left(2 G^\mu + \Re V \right)
		\text{ on } D \cap \Sigma, \end{equation}
		where $\frac{\partial}{\partial n_{\pm}}$
		denote the two normal derivatives to $\Sigma$ inside
		$D \cap \Sigma$.
	\end{enumerate}
\end{theorem}

The proof of Theorem~\ref{theo:QDhigher} can be found in Section~\ref{sec:proof-of-main-result}.
The equality \eqref{eq:S-property} is an analogue of the $S$-property \eqref{Sprop2} in the planar case, see also \cite[Equation (5.30)]{MFR11} and \cite[Definition 2.1]{KS15}.
Note that the equality of the normal derivatives is independent of the coordinate chart.
The identity in \eqref{eq:S-property} is equivalent to the $S$-property defined in Definition~\ref{def:S-property} by the Cauchy-Riemann equations.

\subsection{Overview of the rest of the paper}

In the next section we motivate the theory developed 
in the paper by relating it to certain periodic tiling models.
We focus our discussion on  lozenge tilings with periodic
weightings. 

In Section~\ref{sec:proofs-critical}, we prove Propositions~\ref{prop:CMhigher}~and~\ref{prop:CShigher}, which generalize the results \cite[Lemma~3.1]{MFR11} on critical measures and \cite[Proposition 4.2]{KS15} (see also \cite[Section~9.10]{Rak12}) on critical sets in the plane to compact Riemann surfaces of higher genus.
In Section~\ref{sec:maxmin}, we discuss the max-min energy problem
and prove Theorem~\ref{thm:residues}.
We rely heavily on techniques developed in \cite{KR05,KS15}
and especially \cite{Rak12}.
In Section~\ref{sec:critical-measures-and-quadratic-differentials}, we prove Theorem~\ref{theo:QDhigher}; see in particular Section~\ref{sec:proof-of-main-result}.
The appendix contains  a proof of Proposition~\ref{prop:bipolarGreen} using the Green's function for the Laplacian from Riemannian geometry.

\section{Motivation: periodic tiling models}
\label{sec:motivation}

\subsection{Lozenge tilings of a hexagon}

Our motivation comes from two-dimensional random tiling models with periodic weightings \cite{KOS06}. The main examples are domino tilings of the Aztec diamond \cite{Berg21,CJ16,DK21} and lozenge tilings of a hexagon \cite{Cha20,CDKL20}.  
We focus in this discussion on the latter one. See
the figure for the $ABC$ hexagon and a possible tiling with lozenges
of three types.
The hexagon has vertices at $(0,0)$, $(B,0)$, $(B+C,C)$, $(B+C,A+C)$, $(C,A+C)$ 
and $(0,A)$, where $A,B,C$ are positive integers. The vertices 
of each lozenge in a lozenge tiling have integer coordinates,
see Figure~\ref{fig:ABChexagon}.	

\begin{figure}[t] 
	\centering
	\begin{tikzpicture}[scale=0.4]
		\draw[help lines, color=black!10] (0,0) grid (8,7);
		\draw (0,0) -- (5,0) -- (8,3) -- (8,7) -- (3,7) -- (0,4) -- cycle;
		\draw (-1.2,2) node[right] {$A$}; \draw(7.7,5) node[right]{$A$};
		\draw (3,0) node[below] {$B$}; \draw (5,7) node[above]{$B$};
		\draw (6.5,1.5) node[below]{$C$}; \draw (0.5,5.7) node[right]{$C$};
	\end{tikzpicture} 
	\qquad
	\begin{tikzpicture}[scale=0.4]
		
		\draw[white] (-1,-1) -- (9,9);
		
		\filldraw[draw=black, fill=blue!20] (0,0) -- (1,0) node[above left]{} -- (1,1) -- (0,1) -- cycle;
		\filldraw[draw=black, fill=blue!20] (1,0) -- (2,0) node[above left]{} -- (2,1) -- (1,1) -- cycle;
		\draw[draw=black, fill=yellow!50] (2,0) -- (3,1) -- (3,2) -- (2,1) -- cycle;
		\draw[draw=black, fill=red!20] (2,0) -- (3,0) -- (4,1) -- (3,1) -- cycle;
		\draw[draw=black,fill=red!20] (3,0) -- (4,0) -- (5,1) -- (4,1) -- cycle;
		
		\filldraw[fill=blue!20] (0,1) -- (1,1) -- (1,2) -- (0,2) -- cycle;
		\filldraw[fill=yellow!50] (1,1) -- (2,2) -- (2,3) -- (1,2) -- cycle;
		\filldraw[fill=red!20] (1,1) -- (2,1) -- (3,2) -- (2,2) -- cycle;
		\filldraw[fill=blue!20] (3,1) -- (4,1) -- (4,2) -- (3,2) -- cycle;
		\filldraw[fill=yellow!50] (4,1) -- (5,2) -- (5,3) -- (4,2) -- cycle;
		\filldraw[fill=red!20] (4,1) -- (5,1) -- (6,2) -- (5,2) -- cycle;
		
		\filldraw[fill=yellow!50] (0,2) -- (1,3) -- (1,4) -- (0,3) -- cycle;
		\filldraw[fill=red!20] (0,2) -- (1,2) -- (2,3) -- (1,3) -- cycle;
		\filldraw[draw=black, fill=blue!20] (2,2) -- (3,2) node[above left]{} -- (3,3) -- (2,3) -- cycle;
		\filldraw[fill=yellow!50] (3,2) -- (4,3) -- (4,4) -- (3,3) -- cycle; 
		\filldraw[fill=red!20] (3,2) -- (4,2) -- (5,3) -- (4,3) -- cycle;
		\filldraw[draw=black, fill=blue!20] (5,2) -- (6,2) node[above left]{} -- (6,3) -- (5,3) -- cycle;
		\filldraw[fill=yellow!50] (6,2) -- (7,3) -- (7,4) -- (6,3) -- cycle;
		
		\filldraw[fill=yellow!50] (0,3) -- (1,4) -- (1,5) -- (0,4) -- cycle;
		\filldraw[fill=yellow!50] (1,3) -- (2,4) -- (2,5) -- (1,4) -- cycle;
		\filldraw[fill=red!20] (1,3) -- (2,3) -- (3,4) -- (2,4) -- cycle;
		\filldraw[fill=red!20] (2,3) -- (3,3) -- (4,4) -- (3,4) -- cycle;
		\filldraw[fill=yellow!50] (4,3) -- (5,4) -- (5,5) -- (4,4) -- cycle;
		\filldraw[fill=red!20] (4,3) -- (5,3) -- (6,4) -- (5,4) -- cycle;
		\filldraw[fill=red!20] (5,3) -- (6,3) -- (7,4) -- (6,4) -- cycle;
		\filldraw[fill=yellow!50] (7,3) -- (8,4) -- (8,5) -- (7,4) -- cycle;
		
		\filldraw[fill=yellow!50] (1,4) -- (2,5) -- (2,6) -- (1,5) -- cycle;
		\filldraw[draw=black, fill=blue!20] (2,4) -- (3,4) node[above left]{} -- (3,5) -- (2,5) -- cycle;
		\filldraw[fill=yellow!50] (3,4) -- (4,5) -- (4,6) -- (3,5) -- cycle;
		\filldraw[fill=red!20] (3,4) -- (4,4) -- (5,5) -- (4,5) -- cycle;
		\filldraw[fill=yellow!50] (5,4) -- (6,5) -- (6,6) -- (5,5) -- cycle;
		\filldraw[fill=red!20] (5,4) -- (6,4) -- (7,5) -- (6,5) -- cycle;
		\filldraw[fill=red!20] (6,4) -- (7,4) -- (8,5) -- (7,5) -- cycle;
		
		\filldraw[fill=blue!20] (2,5) -- (3,5) -- (3,6) -- (2,6) -- cycle;
		\filldraw[fill=yellow!50] (3,5) -- (4,6) -- (4,7) -- (3,6) -- cycle;
		\filldraw[fill=yellow!50] (4,5) -- (5,6) -- (5,7) -- (4,6) -- cycle;
		\filldraw[fill=red!20] (4,5) -- (5,5) -- (6,6) -- (5,6) -- cycle;
		\filldraw[fill=blue!20] (6,5) -- (7,5) -- (7,6) -- (6,6) -- cycle;
		\filldraw[fill=blue!20] (7,5) -- (8,5) -- (8,6) -- (7,6) -- cycle;
		
		\filldraw[fill=red!20] (2,6) -- (3,6) -- (4,7) -- (3,7) -- cycle;
		\filldraw[fill=yellow!50] (4,6) -- (5,7) -- (5,8) -- (4,7) -- cycle;
		\filldraw[draw=black, fill=blue!20] (5,6) -- (6,6) node[above left]{} -- (6,7) -- (5,7) -- cycle;
		\filldraw[draw=black, fill=blue!20] (6,6) -- (7,6) node[above left]{} -- (7,7) -- (6,7) -- cycle;
		\filldraw[draw=black, fill=blue!20] (7,6) -- (8,6) node[above left]{} -- (8,7) -- (7,7) -- cycle;
		
		\draw[fill=red!20] (3,7) -- (4,7) -- (5,8) -- (4,8) -- cycle;
		\draw[fill=blue!20] (5,7) -- (6,7) -- (6,8) -- (5,8) -- cycle;
		\draw[fill=blue!20] (6,7) -- (7,7) -- (7,8) -- (6,8) -- cycle;
		\draw[fill=blue!20] (7,7) -- (8,7) -- (8,8) -- (7,8) -- cycle;
	\end{tikzpicture}
	\caption{The $ABC$ hexagon (left) with a random tiling (right). \label{fig:ABChexagon}}
\end{figure}
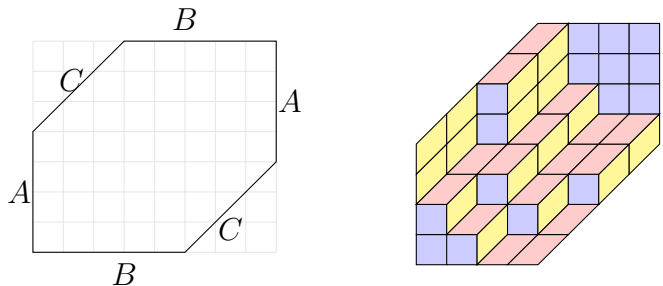

We assign a weight to each lozenge depending on its type and on the 
position it has in the 
hexagon. We denote the weights by  $\lozengeblue(x,y)$, $\lozengeyellow(x,y)$, and $\lozengered(x,y)$, where $(x,y) \in \mathbb Z^2$ is the position of the lower left vertex in case of  $\lozengeblue(x,y)$, $\lozengeyellow(x,y)$, and the position
of the lower right vertex  in case of $\lozengered(x,y)$. 

The weight $w(\mathcal T)$ of a tiling $\mathcal T$ is the
the product of the weights of the lozenges in the tiling, and the probability
of $\mathcal T$ is proportional to its weight 
\begin{equation} \label{eq:probT} 
	\mathbb P(\mathcal T) = \frac{1}{Z} w(\mathcal T),
	\qquad Z = \sum_{\mathcal T} w(\mathcal T). \end{equation}
Two weightings are equivalent if they lead to the same
probabilities \eqref{eq:probT}.
It is known that this model is determinantal, for any choice
of positive weights.

\subsection{Periodic weightings}

The paper \cite{DK21} gives a formula for the correlation kernel
as a double contour integral in case of periodic weighting,
where a weighting is periodic if there exist positive integers $p$ and $q$ such that 
\begin{equation} \label{eq:weightperiodic} 
	\lozengeblue( x+mp,y+nq) = \lozengeblue(x,y),
	\qquad (x,y,  m,n \in \mathbb Z^2) \end{equation}
and similarly for $\lozengeyellow$ and $\lozengered$.
To describe the formula, we need that  
$\lozengered(x,y) = 1$ for every $(x,y) \in \mathbb Z^2$,
which is an assumption one can make without
loss of generality, as for any weighting there is an equivalent
weighting with this property. 
Then for each $x \in \mathbb Z$ we consider a
matrix $T_x$ with entries
\begin{equation} \label{eq:transitionTx} 
	T_x(y,y') = \begin{cases} w_{\begin{tikzpicture}[scale=0.2]
				\draw[fill=blue!20] (0,0) rectangle (1,1);
		\end{tikzpicture}}(x,y),  & \text{ if } y'=y, \\
		w_{\begin{tikzpicture}[scale=0.2]
				\draw[fill=yellow!50] (0,0) -- (1,1) -- (1,2) -- (0,1) -- cycle;
		\end{tikzpicture}}(x,y), & \text{ if } y'=y+1, \\
		0, & \text{ otherwise,} \end{cases} \quad
	\text{ for } (y,y') \in \mathbb Z^2, \end{equation}
that describes a transition
from the horizontal level $x$ to $x+1$. 
Each $T_x$ is a two sided infinite matrix with two non-zero diagonals. 

Consider a periodic weighting with periods $p$ and $q$.
Then $T_x$ is a block Toeplitz matrix with blocks of size 
$q \times q$ because of the periodicity in the vertical direction. 
The symbol of $T_x$ 
is the matrix valued function $z \mapsto A_x(z)$ with
\begin{equation} \label{eq:symbolAx} 
	A_x(z) = \begin{pmatrix} T_x(y,y') \end{pmatrix}_{y,y'= 0}^{q-1} + z \begin{pmatrix} T_x(y,y'+q) \end{pmatrix}_{y,y'=0}^{q-1},
	\qquad z \in \mathbb C. \end{equation}
Due to the periodicity in the horizontal direction 
we have $T_{x+p} = T_x$, and $A_{x+p} = A_x$
for every $x$. Write 
\begin{equation} \label{eq:symbolA} 
	A(z) = A_0(z) A_1(z) \cdots A_{p-1}(z). \end{equation}

The positions of the $\begin{tikzpicture}[scale=0.2]
	\draw[fill=blue!20] (0,0) rectangle (1,1);
\end{tikzpicture}$ and $\begin{tikzpicture}[scale=0.2]
	\draw[fill=yellow!50] (0,0) -- (1,1) -- (1,2) -- (0,1) -- cycle;
\end{tikzpicture}$ lozenges in a random tiling give a random point process 
inside the hexagon. This random point process is determinantal
and a formula for the correlation kernel $K$ is given
(in a more general setting) in \cite[Theorem~4.7]{DK21}.
Let us assume for simplicity that $A=qN$, $C=qM$, $B+C = pL$, 
for certain integers $N,M,L$ where $A,B,C$ are the dimensions of the hexagon. Then 
$K((px,qy),(px',qy'))$ is equal to the $(0,0)$ entry
of the matrix 
\begin{multline} \label{eq:K00entry}
	= - \frac{\chi_{x>x'}}{2\pi i}
	\oint_{\gamma}  A^{x-x'}(z) z^{y'-y} \frac{dz}{z} \\
	+ \frac{1}{(2\pi i)^2}
	\oint_{\gamma} \oint_{\gamma}
	A^{L-x'}(w) R_N(w,z)
	A^x(z) \frac{w^{y'}}{z^{y+1} w^{M+N}} dz dw 
\end{multline}
where the integrals are taken entrywise,
and $R_N$ is a bivariate polynomial of degrees $\leq N-1$
in both variables  that has the reproducing kernel property
\begin{equation} \label{eq:RNreproducing} 
	\frac{1}{2\pi i} \oint_{\gamma} P(w) \frac{A^L(w)}{w^{M+N}} R_N(w,z) dw = P(z)  \end{equation}
for any  matrix valued polynomial $P$ of degree $\leq N-1$.
The contour $\gamma$ in \eqref{eq:K00entry} and \eqref{eq:RNreproducing} is a simple closed contour going around
the origin once in the positive direction.
Other entries of the matrix \eqref{eq:K00entry} give the correlation
kernel at positions $(px,qy+j), (px',qy'+j')$ when $0\leq j,j' \leq q-1$
and there is a similar, but somewhat more complicated
formula 
for the correlation kernel at arbitrary points when the first 
coordinates of the points are not multiples of the period $p$.

The reproducing kernel $R_N$ is related to matrix valued
orthogonal polynomials (MVOPs). The monic MVOP $P_n$ of degree $n$ (if it exists) satisfies
\begin{equation} \label{eq:Pnorthogonal} 
	\frac{1}{2\pi i} \oint_{\gamma} P_n(z)  \frac{A(z)^L}{z^{M+N} } z^k dz
	= H_n \delta_{n,k} \quad \text{ for } 0 \leq k \leq n, \end{equation}
with $\det H_n \neq 0$.
If all MVOP up to degree $N-1$ exist then
\begin{equation} \label{eq:RNsum} R_N(w,z) = \sum_{n=0}^{N-1} P_n^T(w) H_n^{-1} P_n(z). \end{equation}
In our setting it can be shown that the MVOP of degree $N$ exists,
as well as certain other degrees, see also \cite[Proposition 1.1]{GK21}, but not for all degrees up to $N-1$, and the sum
formula \eqref{eq:RNsum} is not valid.  
There is another formula for $R_N$
in terms of the solution of a Riemann-Hilbert problem
which we describe next. 

\subsection{Riemann-Hilbert problem} 

Consider 
\begin{equation} \label{eq:Yzformula} 
	Y(z) = \begin{pmatrix} P_N(z) & \ds \frac{1}{2\pi i} \oint_{\gamma} \frac{P_N(s) 	A(s)^L}{s^{M+N}} \frac{ds}{s-z} \\
		- H_{N-1}^{-1}	P_{N-1}(z) & \ds -\frac{H_{N-1}^{-1}}{2\pi i} \oint_{\gamma} \frac{P_{N-1}(s) 	A(s)^L}{s^{M+N}} \frac{ds}{s-z}  \end{pmatrix},
	\quad z \in \mathbb C \setminus \gamma, \end{equation}
where $P_N$ and $P_{N-1}$ are the monic MVOP of degrees $N$ and $N-1$. 
This matrix valued function of size $2q \times 2q$ is defined and
analytic for $z \in \mathbb C \setminus \gamma$ and satisfies
\begin{equation} \label{eq:Yjump} 
	Y_+(z) = Y_-(z) \begin{pmatrix} I_q & \ds \frac{A(z)^L}{z^{M+N}} \\
		0 & I_q \end{pmatrix}, \qquad z \in \gamma, \end{equation}
where $Y_+(z)$ and $Y_-(z)$ denote the limiting values of $Y(z')$ 
as $z' \to z  \in \gamma$ from the interior and
exterior region, respectively. In addition, $Y$ has
the asymptotic behavior
\begin{equation} \label{eq:Yatinfty}
	Y(z) = (I_{2q}+ O(z^{-1})) \begin{pmatrix} z^N I_q & 0 \\ 0 & z^{-N} I_q
	\end{pmatrix} \quad \text{ as } z \to \infty.
\end{equation}
The Riemann-Hilbert (RH) problem \eqref{eq:Yjump}--\eqref{eq:Yatinfty} 
for MVOP was formulated in \cite{CM12,GIM11} as  
a generalization of the well-known  RH problem for orthogonal
polynomials due to Fokas, Its and Kitaev \cite{FIK92}.
The reproducing kernel is expressed in terms of $Y$ via
the formula \cite{Del10, DK21}
\begin{equation} \label{eq:RNRHproblem}
	R_N(w,z) = \frac{1}{w-z} \begin{pmatrix} 0 & I_q \end{pmatrix}
	Y^{-1}(w) Y(z) \begin{pmatrix} I_q \\ 0 \end{pmatrix}, 
\end{equation}
which can be seen as a matrix valued Christoffel-Darboux formula 
for the sum \eqref{eq:RNsum}.

Of interest is the regime where $N,M,L \to \infty$ in such a
way that their ratios approach certain finite, non-zero limits.
A steepest descent analysis of the RH problem would give asymptotic information
on the reproducing kernel \eqref{eq:RNRHproblem} 
which could then in turn be used
in the double integral  \eqref{eq:K00entry} to find 
limiting behavior of the correlation kernel. This has been done successfully for the case of
domino tilings of the Aztec diamond with $2\times 2$ periodic weightings \cite{DK21}. For periodic hexagon tilings 
with periodicities at most two, the matrix valued
orthogonality can be reduced to scalar orthogonality
on a contour in the complex plane  \cite{Cha20,Cha21, CDKL20,GK21} where notions
from logarithmic potential theory (such as Theorem~\ref{theo:MFR})
apply. The asymptotic analysis of periodic hexagon tilings
with  higher periodicities  remains an open problem, 
since one of the difficulties is to normalize
the RH problem at infinity. For the RH problem for 
orthogonal polynomials on the real line \cite{Dei99,DKMVZ99}, 
this is done with an equilibrium measure in an external field \cite{ST97}, and we are faced with the question of what 
appropriate equilibrium measures in a matrix-valued setting
are when the matrix valued orthogonality cannot be reduced to
scalar orthogonality in the plane.

\subsection{Transformation $Y \mapsto Z$}
To proceed we follow \cite[Section 5.3]{DK21} to
perform a transformation of the RH problem by using the
spectral decomposition $A(z) = E(z) \Lambda(z) E(z)^{-1}$
of \eqref{eq:symbolA}
where $\Lambda(z)$ is a diagonal matrix with the eigenvalues
$\lambda_1(z), \ldots, \lambda_q(z)$ of $A(z)$ on the diagonal
and the columns of $E(z)$ are eigenvectors of $A(z)$. The
transformation 
\begin{equation} \label{eq:Zdef} 
	Z = Y \begin{pmatrix} E & 0 \\ 0 & E \end{pmatrix} \end{equation}
then gives a matrix valued function $Z$ with the jump condition 
\begin{equation} \label{eq:Zjump} Z_+(z) = Z_-(z) \begin{pmatrix} I_q &  
		\ds \frac{\Lambda(z)^L}{z^{M+N}} \\
		0 & I_q \end{pmatrix} \quad z \in \gamma. \end{equation}
The eigenvalues $\lambda_j(z)$ and eigenvectors in $E(z)$ will 
not be entire functions of the $z$ variable and will 
be single-valued only
if we apply certain branch cuts. The definition~\eqref{eq:Zdef} 
will also create a jump of 
$Z$ on these cuts that  takes the form (with an appropriate choice
of the eigenvectors in $E(z)$)
\begin{equation} \label{eq:Zjump2} 
	Z_+ = Z_- \begin{pmatrix} \sigma & 0 \\ 0 & \sigma \end{pmatrix}
	\quad \text{ on branch cuts}, 
\end{equation}
for a permutation matrix $\sigma \in S_q$ that could depend
on the branch cut. 

\subsection{Scalar orthogonality on the Riemann surface} \label{section45}
The characteristic equation 
\begin{equation} \label{eq:speccurve}  
	X : \quad	\det(\lambda I - A(z)) = 0 \end{equation}
is a polynomial equation in $z$ and $\lambda$.  
It is the spectral curve in the sense of \cite{KOS06}
(up to a sign change in $\lambda$ in case $n$ is odd),
see in particular  \cite[p.~1048]{KOS06}. 
The algebraic curve $X$  is a Harnack
curve which implies that it has the maximal number of 
real ovals \cite[Theorem~5.1]{KOS06}. In particular it
has only real branch points when viewed as a $q$-fold
cover of the $z$-plane, and the branch cuts in \eqref{eq:Zjump2}
are on the real line only.
The case where $X$ has genus $\geq 1$ is of special interest as it
allows for the appearance of a smooth (or gas) phase in
the large scale limit,  in addition to the more familiar 
solid and rough (or liquid) phases  \cite{KOS06}.
We show that the matrix valued orthogonality leads to
scalar orthogonality conditions on $X$. See \cite{Ber21a,Ber21b,Ber22,Cha21,FOX20} 
for other recent contributions to orthogonality on a Riemann surface.

Returning to $Z$, we choose a row number $j \in \{1, \ldots, 2q\}$.
Then for  $k=1, \ldots, q$, we consider $Z_{j,k}$
as a function on the $k$th sheet of $X$. The jump
conditions \eqref{eq:Zjump} and \eqref{eq:Zjump2} show that this defines a meromorphic function
$f_j$ on $X$ with a pole at the point(s) at infinity.  
The entry $Z_{j,q+k}$ for $k=1, \ldots, q$
is also considered on the $k$th sheet, and the jump
properties \eqref{eq:Zjump} and \eqref{eq:Zjump2} imply
that they define a holomorphic function $\psi_j$ on
$X \setminus \gamma_X$
(we use $\gamma_X$ to denote the contour on $X$ consisting of a copy of
$\gamma$ on each of the sheets) with the jump
\begin{equation} \label{eq:psijump}
	(\psi_j)_+ = (\psi_j)_- + f_j \frac{\lambda^L}{z^{M+N}} 
	\quad \text{ on } \gamma_X. \end{equation}
The asymptotic condition \eqref{eq:Yatinfty} and the
definition~\eqref{eq:Zdef} imply that $f_j$ has a pole
and $\psi_j$ has a zero at each point at infinity. The precise
order of the poles and zeros depends on the possible branching
of $X$ at infinity, and 
the behavior of $E(z)$ at infinity. In any case the conditions \eqref{eq:psijump}
imply certain orthogonality conditions for $f_j$ of the
form
\begin{equation} \label{eq:orthogonality} 
	\int_{\gamma_X} f_j \frac{\lambda^L}{z^{M+N}} \omega  = 0, 
\end{equation}
for a large class of meromorphic differentials $\omega$ on $X$
with poles at infinity.

Closer inspection shows that the pole conditions  
at infinity imply that $f_j$ belongs to a vector space 
of meromorphic functions of dimension $\approx qN$ and the orthogonality
conditions \eqref{eq:orthogonality} are for $\omega$
in a vector space of meromorphic differentials of
dimension $\geq qN-c$ with a constant $c$ that is independent of $N,M,L$. Typically  \eqref{eq:orthogonality} will not be enough
to characterize $f_j$ as an orthogonal meromorphic function on $X$,
as the conditions \eqref{eq:orthogonality} will have to
be supplemented by $\approx c$ additional conditions.

Since $c$ is independent of $N,M,L$ we expect that in 
generic cases the additional conditions do
not influence the behavior of the bulk of the zeros of $f_j$
in the large $N,M,L$ limit. We expect that the zeros
will accumulate on an $S$-curve in an external field
on the Riemann surface $X$, as it is known in the genus zero case.
The equilibrium measure in external field on an $S$-curve
has a $g$-function that will be useful in the next transformation
in the  steepest descent analysis of the Riemann-Hilbert problem.

To illustrate, let us take $q=2$ and assume that $X$ has 
branching at infinity.
Let $g_1$ be the restriction of the $g$ function to the first sheet,
and $g_2$ to the second sheet. 
Then we expect that the following transformation $Z \to U$,
\begin{equation} \label{eq:ZtoU} 
	U = Z \diag \left( e^{-2Ng_1}, e^{-2Ng_2}, e^{2Ng_1}, e^{2Ng_2} \right) 
\end{equation}
will be the appropriate next step in the steepest descent analysis.

\subsection{The two-periodic Aztec diamond}
   
The analogous model of the domino tilings of the Aztec diamond
 with two-periodic weightings was successfully analyzed in \cite{DK21}
but without the use of $g$-functions. 
The transformation $Z \to U$ in \cite[formula (5.20)]{DK21} (note that $X$ is
used in \cite{DK21} instead of $Z$) is explicit in terms of $\lambda$ and $z$.
Let us show here that the transformation is actually of the form
\eqref{eq:ZtoU}.

The Riemann surface $X$ in \cite{DK21} is
associated with	the equation
	\[ X : \qquad w^2 = z (z-x_1)(z-x_2) \]
	with real $x_1, x_2$ given by
	$x_1 = - \alpha^{2}$, $x_2 = - \alpha^{-2}$ for some $\alpha > 1$.
	A point on $X$ is denoted by $p = (w,z)$ and we write
	$w = w(p)$, $z= z(p)$, for the projections onto the $w$ 
	and $z$ coordinate.
	Then $X$ is a two-sheeted compact Riemann surface with
	two sheets $\mathbb C \setminus (-\infty, x_1] \cup [x_2,0])$.
	The first sheet is such that $w(p) > 0$ for $p = (w,z)$
	with $z(p) > 0$ on the first sheet. 
	
The meromorphic function $\lambda$ on $X$ has 
a double pole at $p_1$ where $p_1$ is the point on
the first sheet with $z(p_1) = 1$, and a double zero
at $p_2$, the point on the second sheet with $z(p_2) = 1$.
There are no other zeros or poles \cite[Lemma~5.2]{DK21}.
	The bipolar Green's function $G(p,p_1)$ with pole
	at infinity and $p_1$ is explicit
	in terms of $\lambda$, namely for some constant $c$,
	\begin{equation} \label{eq:Gpp1} 
		G(p,p_1) = \frac{1}{4} \log |\lambda(p)| - \frac{1}{2} \log |z(p) - 1| + c \end{equation}
	since indeed the right-hand side satisfies the 
	requirements of Proposition~\ref{prop:bipolarGreen}
	(a), (b) and (c). Note that \eqref{eq:Gpp1} is finite
	at $p=p_2$ since
	$\lambda$ has a double zero at $p_2$.
	Without loss of generality we may assume that the bipolar
	Green's function is taken such that $c=0$ in the above formula. 
	
	Let $0 < r \leq 1$ and let $\mu_r$ be the uniform normalized Lebesgue measure on
	the circle $F_r :|z-1| = r$ on the first sheet. By conformal
	invariance it is the balayage of the point mass $\delta_{p_1}$
	onto $F_r$, which means that $\int G(p,q) d\mu_r(q) = G(p,p_1)$
	for $p \in X$ in the exterior $\Omega_{\text{ext}}$ 
	of the circle (which includes the entire second sheet).
	The equality extends to $F_r$, and thus 
	$2 \int G(p,q) d\mu_r(q) =  \frac{1}{2} \log |\lambda(p)| -  \log r$
	on the support of $\mu_r$, which means that $\mu_r$
	is the equilibrium measure of $F_r$ in
	the external field $- \frac{1}{2} \log |\lambda|$,
	cf.\ Proposition~\ref{prop:eqmeasure}.
	
	We verify that $F_r$ has the $S$-property in the external field.
	Since
	$\frac{1}{4} \log |\lambda(p)| + \frac{1}{2} \log |z(p) - 1|$ is harmonic in the interior domain $\Omega_{\text{int}}$ 
	of $F_r$ (including at $p_1$) 
	and agrees with $G(p,p_1) + \log r$ on the circle. We conclude
	\[ \int G(p,q) d\mu_r(q) = 
	\begin{cases} \frac{1}{4} \log |\lambda(p)| - \frac{1}{2} \log |z(p) - 1|,  & \text{ in } \Omega_{ext}, \\[5pt]
		\frac{1}{4} \log |\lambda(p)| + \frac{1}{2} \log |z(p) - 1|	- \log r,
		& \text{ in } \Omega_{int}.  \end{cases}
	\]
	Then, with appropriate branches of the logarithm
	\[ g(p) = \begin{cases}
		- \frac{1}{4} \log \lambda(p) + \frac{1}{2} \log (z(p) - 1), & \text{ in } \Omega_{ext} \\[5pt]
		- \frac{1}{4} \log \lambda(p) - \frac{1}{2} \log (z(p) - 1) - \log r, & \text{ in } \Omega_{int} \end{cases} \]
	and indeed $g_+ + g_- + \frac{1}{2} \log \lambda$
	is constant on $\supp(\mu)$, which is the $S$-property in
	external field by Definition~\ref{def:S-property}.
	It follows that the transformation in \cite[Section 5.4]{DK21}
	is indeed of the form \eqref{eq:ZtoU}.

The $2 \times k$ periodic Aztec diamond was studied by
Berggren \cite{Berg21} using a Wiener-Hopf factorization
technique developed in \cite{BD19}. 
The formula \eqref{eq:Gpp1}  applies in the $2 \times k$ periodic case as well and it gives rise to an equilibrium measure 
in the same way as described above.

\subsection{Hexagon tilings}

In the Aztec diamond example, one is fortunate to be able to
find a contour with the $S$-property via an explicit
construction. In the hexagon tiling models this does not
seem to work and that is why we redeveloped the methods of
this paper in order to prove the existence of contours with
the $S$-property in an external field. 

The final result of this paper, Theorem~\ref{theo:QDhigher}, is
restricted to a special situation of genus one. It will apply
to lozenge tilings with periodic weightings of periods $p=3$
and $q=2$. In that case the spectral curve 
\eqref{eq:speccurve} has genus one and it can be put
in the form
$X : \quad w^2 = (z-x_1)(z-x_2)(z-x_3)$ with
$x_1 < x_2 < x_3 \leq 0$. If $x_3 = 0$ then we are in the situation
of  \eqref{Xgenus1}. The external field is
\begin{equation} \label{eq:varphihexagon} 
	\varphi(z) = \Re V(z) = 
	- \frac{b}{2} \log |\lambda| + \frac{1+c}{2} \log |z| 
	\end{equation}
with $b = \lim \frac{L}{N}$, $c = \lim \frac{M}{N}$,
$3b > 2c$,
which is invariant under the involution \eqref{involution}.
The condition $3b > 2c$ comes from the fact that
$3L = B+C > C = 2M$, hence $3 \frac{L}{N} > 2 \frac{M}{N}$,
and we take the large $N$ limit such that the strict
inequality remains valid. Then $dV$ has residue 
$r_0 = 1+c$ at $0$ (due to the branching at $0$). Since $\lambda$
will have a third order pole at $\infty$ and $z$ has a
second order pole there, one may
calculate from \eqref{eq:varphihexagon} that $dV$ has
a simple pole at $\infty$ with residue 
$r_{\infty} = 3 \frac{b}{2} - 2 \frac{1+c}{2} = \frac{3 b-2c}{2} - 1$.   Thus the residue conditions of Theorem~\ref{thm:residues}
are satisfied, and thus there is a critical set $F$
with equilibrium measure $\mu^F$ that is a critical measure
in the external field $\varphi$.
By symmetry we may assume that $F$ and $\mu^F$ are
invariant under the involution $\sigma$.

If the support of $\mu^F$ does not intersect the bounded real oval
$C_1$, then by Theorem~\ref{theo:QDhigher}, the support  is a union of maximal trajectories of a quadratic differential $-\upomega$ given by \eqref{QDhigher2}.
Note that by \eqref{QDhigher1}, $\upomega$ has double poles at the poles of $dV$, but also at $p_\infty$ due to the simple pole of $\int C(u,q) d \mu(q)$ at $u = p_\infty$.
Counting multiplicities, $\upomega$ thus has a total of $8$ poles.
Because we are in genus one, we find that $\upomega$ then also has $8$ zeros (counted with multiplicity).

The  equilibrium measure in external field
and its associated $g$-function
will be useful in the steepest descent analysis of the
RH problem as we indicated.
The details of the steepest descent analysis, and their consequences
for the 
random tiling model are under current investigation.

\section{Proofs of Propositions~\ref{prop:CMhigher}~and~\ref{prop:CShigher}}\label{sec:proofs-critical}

\subsection{Proof of Proposition~\ref{prop:CMhigher}}
\begin{proof}
	Let $\Phi$ denote the flow associated with the $C^1$ vector field
	$h$.
	We have 
	\begin{equation} \label{CMhigher5}
		\left. \frac{d}{d\varepsilon} \right|_{\varepsilon=0} V(\Phi(\varepsilon,p)) = h(p) dV(p) \end{equation} and thus by
	\eqref{CMhigher1}
	\begin{align} \nonumber \int V d\mu_{\varepsilon,h} & = \int V(\Phi(\varepsilon,p)) d\mu(p) \\ \label{CMhigher6}
		& =\int V d\mu 
		+ \varepsilon \int h dV  d\mu + o(\varepsilon) 
		\quad \text{ as } \varepsilon \to 0. \end{align}	
	
	Since $p \mapsto G(p,q)$ is harmonic, it is locally the real part of
	a holomorphic function $\mathcal G(p,q)$ and
	\[ C(p,q) = -2 \partial_p G(p,q) = - \partial_p \mathcal G(p,q). \]
	Then for $p \neq q$,
	\[ \left.\frac{d}{d\varepsilon} \right|_{\varepsilon=0} \mathcal G(\Phi(\varepsilon,p), q) 
	= - h(p) C(p,q)  \]
	and by taking the real part
	\[ \left. \frac{d}{d\varepsilon}\right|_{\varepsilon=0} G(\Phi(\varepsilon,p),q) = 
	- \Re (h(p) C(p,q) ). \]
	Since $G$ is symmetric in the two variables, we then also have
	\[ \left. \frac{d}{d\varepsilon}\right|_{\varepsilon=0} G(p,\Phi(\varepsilon,q)) = 
	- \Re (h(q) C(q,p) ), \]
	and by the chain rule 
	\begin{equation} \label{CMhigher7}  	\left. \frac{d}{d\varepsilon}\right|_{\varepsilon=0} G(\Phi(\varepsilon,p),\Phi(\varepsilon,q)) = 
		- \Re (h(p) C(p,q) + h(q) C(q,p) ) \end{equation}
	Then because of \eqref{GreenE1}, \eqref{CMhigher1} and \eqref{CMhigher7}.
	\begin{multline} 
		E[\mu_{\varepsilon,h}] = \iint G(\Phi(\varepsilon,p), \Phi(\varepsilon,q)) d\mu(p) d\mu(q) \\  \label{CMhigher8}
		= E[\mu]  
		- \varepsilon \Re \iint (h(p) C(p,q) + h(q) C(q,p)) d\mu(p) d\mu(q) 
		+ o(\varepsilon) 
	\end{multline}
	as $\varepsilon \to 0$. 
	Combining \eqref{CMhigher6} and 
	\eqref{CMhigher8} we obtain the limit in part (a). 
	\medskip
	
	If $\mu$ is critical then the expression in part (a) is zero for every $h$,
	and this implies that the real part of $D_{V,h}(\mu)$ vanishes. 
	Changing $h$ to $ih$,
	we find that the imaginary part vanishes as well,
	and part (b) follows.
\end{proof}

\subsection{Convergence of perturbed measures in energy norm}

The next lemma concerns the convergence of the measures in energy norm; see \eqref{energynorm}.

\begin{lemma} \label{lem:lemma45}
	Suppose $\mu$ is compactly supported in $X \setminus \{p_{\infty}\}$ with $E[\mu] < +\infty$. Let $h$ be 
	a $C^1$ vector field. Then
	the measures $\mu_{\varepsilon, h}$ tend to $\mu$
	as $\varepsilon \to 0$ in energy norm.
\end{lemma}
\begin{proof}	
	If $f$ is a continuous function, then
	\[ \int fd\mu_{\varepsilon,h} = \int f(\Phi(\varepsilon,p))d\mu(p)\]
	and $f(\Phi(\varepsilon,p)) \to f(p)$ for every $p$.
	Since $f$ is bounded on $K$, we can apply dominated
	convergence to conclude that $\int f d\mu_{\varepsilon,h} \to \int f d\mu$ as $\varepsilon \to 0$. This proves
	the weak$^*$ convergence of $(\mu_{\varepsilon,h})_\varepsilon$
	to $\mu$.   
	
	Since $q \mapsto \int G(p,q) d\mu(p)$ is  lower semi-continuous, we have by weak$^*$ convergence
	\[ E[\mu] = \iint G(p,q) d\mu(p) d\mu(q)
	\leq \liminf_{\varepsilon \to 0}
	\iint G(p,q) d\mu(p) d\mu_{\varepsilon,h}(q). \]
	From \eqref{CMhigher8} we conclude 
	$E[\mu_{\varepsilon,h}] \to E[\mu]$, and therefore
	\begin{align*} 
		\limsup_{\varepsilon \to 0} E[\mu_{\varepsilon,h} - \mu]
		& = \limsup_{\varepsilon \to 0} 
		\left( E[\mu_{\varepsilon,h}] + E[\mu]
		- 2 \iint G(p,q) d\mu(p) d\mu_{\varepsilon,h}(q) \right)
		\\ & \leq E[\mu] +  E[\mu] - 2 E[\mu] = 0,
	\end{align*}
	which is the convergence in energy norm.
\end{proof}

For later use, we need a uniformity in the limit \eqref{CMhigher4}
for measures on a fixed compact. 
\begin{lemma} \label{lem:lemma46}
	Let $K \subset X \setminus \{p_{\infty}\}$ be
	a compact set such that $\varphi = \Re V$ is bounded on $K$.
	Let $h$ be a $C^1$ vector field.
	Then for every  probability measure $\mu$ 
	on $K$ with $E[\mu] < + \infty$, we have	
	\begin{equation} \label{CMhigher9}
		E_\varphi[\mu_{\varepsilon,h}] - E_\varphi[\mu]
		- \varepsilon \Re D_{V,h}(\mu)  = o(\varepsilon)
		\quad \text{ as } \varepsilon \to 0, \end{equation}
	where the $o(\varepsilon)$
		only depends on $h$ and $K$, but not on $\mu$.
\end{lemma}
\begin{proof} 
	Let $\Phi$ be the flow associated with the vector field $h$.
	Since $\varphi$ is bounded on $K$, the set $K$
	does not contain any of the poles of $dV$.
	Then there is $\varepsilon > 0$ small enough such
	that 
	\begin{equation} \label{CMhigher10} \bigcup_{t \in [-\varepsilon, \varepsilon]} \bigcup_{p \in K}
		\Phi(t, p)  \end{equation}
	is a compact subset of $X \setminus \{p_{\infty}\}$ that
	does not contain  any of the poles of $dV$.
	Thus $\varphi$ is bounded on the set \eqref{CMhigher10}.
	
	We already observed in Remark \ref{remarkDVh} that 
	\begin{equation} \label{CMhigher11}
		(p,q) \mapsto H(p,q) = h(p) C(p,q) + h(q) C(q,p)
	\end{equation} 
	is a well
	defined function on $(X \setminus \{p_{\infty}\})^2$.
	It is clearly $C^1$ function for $p \neq q$, since
	we assumed that $h$ is $C^1$.
	
	Suppose $p,q$ are in a local chart, where $p$ and $q$
	correspond to $z$ and $w$ respectively in the local coordinate. 
	Then
	\[ C(p,q) = \frac{1}{z-w} dz + \widetilde{C}(z,w) dz \]
	where $\widetilde{C}(z,w)$ is $C^{\infty}$ smooth.
	Write $h = h_1 \frac{\partial}{\partial z} +
	h_2 \frac{\partial}{\partial \bar{z}}$
	with $C^1$ functions $h_1$ and $h_2$. Then by \eqref{CMhigher11}
	\begin{align*} H(p,q)  & = h_1(z) \left( \frac{1}{z-w} + \widetilde{C}(z,w) \right) + h_1(w) \left(\frac{1}{w-z} + 
		\widetilde{C}(w,z) \right) \\ 
		& = \frac{h_1(z)- h_1(w)}{z-w}
		+ h_1(z) \widetilde{C}(z,w) + h_1(w) \widetilde{C}(w,z)
	\end{align*}
	which shows that $H$ is a continuous function in a local chart.
	Therefore $H$ is uniformly continuous on the compact
	set \eqref{CMhigher10}.
	It follows from the uniform continuity that
	\begin{equation}\label{CMhigher12}
		\sup_{|t| \leq \varepsilon} \sup_{p,q \in K} |H(\Phi(t,p), \Phi(t,q)) - H(p,q)| = o(1)
	\end{equation}
	as $\varepsilon \to 0$, where 
	the $o$-term only depends on $h$ and the compact set $K$.
	
	Next, by integrating \eqref{CMhigher7} and recalling
	\eqref{CMhigher11}, we find
	\begin{multline*}
		G(\Phi(\varepsilon,p), \Phi(\varepsilon,q))
		- G(p,q) + \varepsilon \Re H(p,q)  \\
		= -  \Re 
		\int_0^\varepsilon \left[ H(\Phi(s,p), \Phi(s,p)) 
		- H(p,q) \right] ds.
	\end{multline*} 
	Combining this with \eqref{CMhigher12}, we
	arrive at the estimate  
	\[ \left|
	G(\Phi(\varepsilon,p), \Phi(\varepsilon,q))
	- G(p,q) + \varepsilon \Re H(p,q) \right|
	= o(\varepsilon), \quad p, q \in K, \]
	as $\varepsilon \to 0$, where the $o$-term only depends on $h$ and $K$.
	We integrate with respect to $d\mu(p) d\mu(q)$ where
	$\mu$ is a probability measure on $K$ with
		$E[\mu] < \infty$, to find that as $\varepsilon \to 0$,
	\begin{multline} \label{CMhigher13} \left| \iint G(p,q) d\mu_{\varepsilon,h}(p) d\mu_{\varepsilon,h}(q)
		- \iint G(p,q) d\mu(p) d\mu(q) \right. \\ \left.
		+ \varepsilon \Re \iint H(p,q) d\mu(p) d\mu(q)
		\right| = o(\varepsilon), \end{multline}
	with the $o$-term independent of $\mu$.
	
	Using \eqref{CMhigher5} and the fact that $\Re V$
	is bounded on \eqref{CMhigher10}  we find in an analogous way that for any probability
	measure $\mu$ on $K$, as $\varepsilon \to 0$,
	\begin{equation} \label{CMhigher14}
		\left| \int \Re V d\mu_{\varepsilon,h} -
		\int \Re V d\mu - \varepsilon \int hdV d\mu \right|
		= o(\varepsilon). \end{equation}
	Combining \eqref{CMhigher13} and \eqref{CMhigher14},
	and recalling the definitions, 
	we obtain \eqref{CMhigher9}.
\end{proof}

\subsection{Proof of Proposition~\ref{prop:CShigher}}

\begin{proof} 	
	Let $F \subset X \setminus \{p_{\infty}\}$ be a union of continua such that $\Re V$ is
	bounded from below on $F$ and let $\mu^F$ be the equilibrium measure of $F$ in the external field $\varphi$.
	
	Part (b) follows directly from the definition once part (a) has been established, so we only need to prove part (a).
	
	To establish \eqref{CShigher3} it suffices to show that
	\begin{equation} \label{CShigher4} 
		E_{\varphi}(F_{\varepsilon,h}) - E_{\varphi}(F)
		= E_{\varphi}[\mu_{\varepsilon,h}^F] - E_{\varphi}[\mu]
		+ o(\varepsilon) \end{equation}
	as $\varepsilon \to 0$, since we already know from 
	Proposition~\ref{prop:CMhigher} that the limit in the
	right-hand side of \eqref{CShigher3} exists, and having
	\eqref{CShigher4} we conclude that the limit in the
	left-hand side also exists and the equality \eqref{CShigher3} holds.
	
	For a $C^1$ vector field $h$, and $\varepsilon \in \mathbb R$, 
	we are going to use the four probability measures
	$\mu^F$, $\mu^F_{\varepsilon,h}$, $\mu^{F_{\varepsilon,h}}$
	and $(\mu^{F_{\varepsilon,h}})_{-\varepsilon,h}$. 
	To simplify notation, we write  
	\[ \mu = \mu^F, \quad \mu_{\varepsilon} = \mu^{F}_{\varepsilon,h},
	\quad \mu^{\varepsilon} = \mu^{F_{\varepsilon,h}}, \quad \mu^{\varepsilon}_{-\varepsilon} = (\mu^{F_{\varepsilon,h}})_{-\varepsilon,h}. \]
	
	Thus $\mu^{\varepsilon}$ is the equilibrium measure of $F_{\varepsilon, h}$ (see \eqref{CShigher1}) in the external field $\varphi$,
	and therefore  
	\begin{equation} \label{CShigher5}  E_{\varphi}(F_{\varepsilon,h}) = E_{\varphi}[\mu^{\varepsilon}] \leq E_{\varphi}[\mu_\varepsilon], \end{equation}
	since $\mu_{\varepsilon}$ is supported on $F_{\varepsilon,h}$.
	Since $\mu^{\varepsilon}_{-\varepsilon}$ is the image
	of $\mu^{\varepsilon}$ under the reverse flow associated
	with $h$, it is supported on $F$, and hence
	\begin{equation} \label{CShigher6} 
		E_{\varphi}(F) = E_{\varphi}[\mu] \leq E_{\varphi}[\mu^{\varepsilon}_{-\varepsilon}], \end{equation}
	
	Applying Lemma~\ref{lem:lemma46} to $\mu$ and $\mu^{\varepsilon}$
	(with $-\varepsilon$ instead of $\varepsilon$ for the latter
	measure), we obtain
	\begin{equation} \label{CShigher7} 
		\begin{aligned} 
			E_{\varphi}[\mu_\varepsilon] & = E_{\varphi}[\mu]
			+ \varepsilon \Re D_{V,h}(\mu) + o(\varepsilon), \\
			E_{\varphi}[\mu^{\varepsilon}] & =
			E_{\varphi}[\mu^{\varepsilon}_{-\varepsilon}]
			+\varepsilon \Re D_{V,h}\left(\mu^{\varepsilon}\right) + o(\varepsilon).
		\end{aligned}
	\end{equation}
	Subtracting the two identities in \eqref{CShigher7} 
	and using the inequalities \eqref{CShigher5} and \eqref{CShigher6}, 
	we get
	\begin{equation} \label{CShigher8} 0  \leq E_{\varphi}[\mu_\varepsilon] -	E_{\varphi}[\mu^{\varepsilon}]   \leq \varepsilon \Re D_{V,h}(\mu)
		-\varepsilon \Re D_{V,h}\left(\mu^{\varepsilon}\right)
		+ o(\varepsilon) 
	\end{equation}
	The expression \eqref{CMhigher3} shows that
	$D_{V,h}(\mu^{\varepsilon})$ remains bounded
	as $\varepsilon \to 0$ and therefore by \eqref{CShigher8} 
	\begin{equation} \label{CShigher9} 
		\lim_{\varepsilon \to 0} \left( E_{\varphi}[\mu_{\varepsilon}]
		- E_{\varphi}[\mu^{\varepsilon}]\right) = 0. \end{equation}
	
	Next, we use the identity
	\[ E[\mu_\varepsilon - \mu^{\varepsilon}]
	= 2E_{\varphi}[\mu_{\varepsilon}]	
	+ 2 E_{\varphi}[\mu^{\varepsilon}]
	- 4 E_{\varphi}\left[ \frac{\mu_{\varepsilon}+ \mu^{\varepsilon}}
	{2} \right] 
	\]
	and the inequality
	\[ E_{\varphi}(F_{\varepsilon,h}) \leq  E_{\varphi}\left[ \frac{\mu_{\varepsilon}+ \mu^{\varepsilon}}
	{2} \right], \]
	which holds since $\frac{\mu_{\varepsilon}+ \mu^{\varepsilon}}
	{2}$ is a probability measure on $F_{\varepsilon,h}$, to
	find that
	\[  0 \leq E[\mu_{\varepsilon} - \mu^{\varepsilon}]
	\leq 2 E_{\varphi}[ \mu_{\varepsilon}]
	- 2 E_{\varphi}[\mu^{\varepsilon}]. \]
	Hence in view of \eqref{CShigher9}
	\begin{equation} \label{CShigher10} 
		\lim_{\varepsilon \to 0} E[\mu_{\varepsilon} - \mu^{\varepsilon}] = 0. \end{equation}
	
	Since $\mu_{\varepsilon}$ tends to $\mu$ in energy norm,
	see Lemma~\ref{lem:lemma45}, 
	we find from \eqref{CShigher10} that $\mu^{\varepsilon}$ also
	tends to $\mu$ in energy norm, and in particular 
	in weak$^*$ sense. Then
	\begin{equation} \label{CShigher11} \lim_{\varepsilon \to 0} D_{V,h}(\mu^{\varepsilon}) = D_{V,h}(\mu) \end{equation}
	by weak$^*$ convergence, since the functions that
	appear in the integrals \eqref{CMhigher3}  are 
	bounded and continuous.
	Using \eqref{CShigher11} in \eqref{CShigher8}, we
	find
	\[ E_{\varphi}[\mu_{\varepsilon}] - E_{\varphi}[\mu^{\varepsilon}] = o(\varepsilon) \]
	as $\varepsilon \to 0$. 
	Combining this with the equalities in \eqref{CShigher5} and \eqref{CShigher6},
	we obtain \eqref{CShigher4}, and
	the proof of Proposition~\ref{prop:CShigher} is complete.
\end{proof}

\section{Proof of Theorem~\ref{thm:residues}}
	\label{sec:maxmin}
	
\subsection{A preliminary lemma}

\begin{lemma} \label{lem:prelim}
	Suppose $F \in \mathcal F_{\delta}$, and let $\widetilde{F}$
	be a connected component of $F$.
	\begin{enumerate}
		\item[\rm (a)] If $\widetilde{F} \subset U_{\delta_0}(p_0)$
		then $\widetilde{F}$ separates $p_0$ from 
		$\partial U_{\delta_0}(p_0)$.
		\item[\rm (b)] If $\widetilde{F} \subset U_{\delta_{\infty}}(p_\infty)$
		then $\widetilde{F}$ separates $p_\infty$ from 
		$\partial U_{\delta_\infty}(p_\infty)$.
		\item[\rm (c)] There is $\eta > 0$, depending only on $\delta$,
		such that $\diam (\widetilde{F}) \geq \eta$.
	\end{enumerate} 
\end{lemma}
\begin{proof}
	Any contour $\gamma \in \mathcal S_{\delta}$
	has the properties (a) and (b), since otherwise
	$\gamma$ would be homotopic to a point in $X \setminus \{p_0,p_{\infty}\}$, which would contract the definition
	of $\mathcal S_{\delta}$ in Definition~\ref{def:Tdelta} (a).
	
	We next prove that 
	\begin{equation} \label{eq:Sdeltadiam} 
		\inf_{\gamma \in \mathcal S_{\delta}} \diam(\gamma) > 0. \end{equation}
	By compactness of $X$ there is an $\eta > 0$ such that
	every $\eta$ neighborhood of a point is contractible.
	Thus any $\gamma \in S_{\delta}$ that is not contractible in $X$
	must have $\diam(\gamma) \geq \eta$.
	Now suppose $\gamma \in \mathcal S_{\delta}$ is 
	contractible in $X$. Then $X \setminus \gamma$ has a simply connected component
	that contains at least one of $p_0, p_{\infty}$ since otherwise
	$\gamma$ would be  contractible in $X \setminus \{p_0, p_{\infty}\}$
	which again would contradict Definition~\ref{def:Tdelta} (a).  
	If the simply connected component contains both $p_0$ and $p_{\infty}$
	then $\gamma$ will have a minimal diameter.
	If the simply connected component contains only one of $p_0$ and $p_{\infty}$,
	then $\gamma$ separates $p_0$ from $p_{\infty}$
	and it will have a minimum diameter since
	$\gamma$ is outside of $U_{\delta}(p_0)$ and $U_{\delta}(p_{\infty})$.
	Thus \eqref{eq:Sdeltadiam} holds.
	
	Thus properties (a), (b), and (c) hold if $F = \gamma \in \mathcal S_{\delta}$.
	Then by \eqref{eq:defTdelta} it also holds if $F = \gamma \in \mathcal T_{\delta}$.
	 The properties are preserved under taking closure in the Hausdorff metric and the lemma follows.	
\end{proof}

\subsection{Estimates of $E_{\varphi}(F)$}
We start with a proposition, whose proof should be compared 
to that of \cite[Proposition 5.1]{KS15}.
\begin{proposition} \label{prop:residues}
	Suppose $\varphi$ has the behavior
	\eqref{phiatp0} near $p_0$ and the behavior
	\eqref{phiatpinfty} near $p_{\infty}$. 
	Let $\mathcal F$, and $\delta_0, \delta_{\infty}$ be as in Definition~\ref{def:Tdelta}.
	Then there is a constant $C$ such that
	the following hold.
	\begin{enumerate} 
		\item[\rm (a)] 
		If $F \in \mathcal F$ intersects $\partial U_{\delta}(p_0)$ 
		for some $\delta < \delta_0/2$
		then  \begin{equation} \label{EFestimate1}
			E_{\varphi}(F) \leq (r_0-1) \log \delta + C. 
		\end{equation}
		\item[\rm (b)] 
		If $F \in \mathcal F$ intersects $\partial U_{\delta}(p_{\infty})$
		for some $\delta < \delta_{\infty}/2$, then 
		\begin{equation} \label{EFestimate2} 
			E_{\varphi}(F) \leq (1+r_{\infty})   \log \delta + C. \end{equation}
	\end{enumerate}
\end{proposition}
\begin{proof}
	(a) Take  $0 < \delta < \delta_0/2$ and assume $F \in \mathcal F$ intersects $\partial U_{\delta}(p_0)$. 
	Let $\widetilde{F}$ be a connected component of $F$ that
	intersects $\partial U_{\delta}(p_0)$.
	Then either $\widetilde{F}$ is fully contained in $U_{2\delta}(p_0)$
	and we put $F_0 = F$,
	or it contains a sub-continuum $F_0 \subset \widetilde{F}$
	that intersects both $\partial U_{\delta}(p_0)$ and $\partial U_{2\delta}(p_0)$ and that is contained in $\overline{U_{2\delta}(p_0)}$.
		
	In both cases $z_0(F_0)$ is a continuum in $\mathbb C$ and we claim that
	\begin{equation} \label{diamF0} 
		\diam(z_0(F_0)) \geq \delta, \end{equation}
	where $z_0$ is the local coordinate at $p_0$.
	The claim \eqref{diamF0} follows immediately if
	$F_0$ intersects both $\partial U_{\delta}(p_0)$
	and $\partial U_{2\delta}(p_0)$, since in that case the
	$z_0$-images of the two intersection points are in
	$z_0(F_0)$ with distance $\geq \delta$.
	The claim \eqref{diamF0} also follows 
	if $F_0 = \widetilde{F} \subset U_{2\delta}(p_0)$,
	since then $z_0(F_0)$ is a continuum in $\mathbb C$ 
	that intersects $\partial D(0,\delta)$ and  
	separates $p_0$ from $p_\infty$ by Lemma~\ref{lem:prelim}. 
	
	Let $\omega$ be the equilibrium measure from usual
	logarithmic potential theory for the continuum $z_0(F_0)
	\subset \overline{D(0,2\delta)}$.
	From \eqref{diamF0} it follows that the capacity of
	$z_0(F_0)$ is at least the capacity of a straight line
	segment of length $\delta$, which is $\delta/4$, see e.g.\ \cite[Theorem~5.3.2]{Ran95}.
	Therefore 
	\begin{equation} \label{logestimate1} 
		\iint \log \frac{1}{|z-w|} d\omega(z) d\omega(w) 
		\leq -\log \frac{\delta}{4}. \end{equation}
	
	Because of \eqref{Gnearp0} we have for some constant $C_1$
	\begin{equation} \label{Gpqestimate} G(p,q) \leq \log \frac{1}{|z_0(p)-z_0(q)|} + C_1,
		\qquad p,q \in U_{\delta_0}(p_0) \end{equation}
	Let $\mu$ be the pullback of $\omega$ by the 
	local coordinate $z_0$.
	Then $\mu$ is a probability measure on $F \cap  \overline{U_{2\delta}(p_0)} \subset U_{\delta_0}(p_0)$
	so that   by \eqref{Gpqestimate}, the properties
	of the pullback measure,  and \eqref{logestimate1}
	\begin{align} \nonumber
		\iint G(p,q) d\mu(p) d\mu(q) 
		& \leq \iint \log \frac{1}{|z_0(p)-z_0(q)|} d\mu(p)
		d\mu(q) + C_1 \\ \nonumber
		& = \iint \log \frac{1}{|z-w|} d\omega(z) d\omega(w) 
		+ C_1 \\ \label{Gpqestimate2} 
		& \leq  -\log \delta + \log 4 + C_1. \end{align}
	Because of \eqref{phiatp0} there is $C_2$
	such that
	\[	\varphi(p) \leq r_0 \log |z_0(p)| + C_2, \qquad 
	\text{ for } p \in U_{\delta_0}(p_0). \]
	Since $|z_0(p)| \leq 2 \delta$ for $z \in \supp(\mu) \subset
	\overline{U_{2\delta}(p_0)}$, we thus find 
	\begin{equation} \label{Gpqestimate3} 
		\int \varphi(p) d\mu(p) \leq r_0 \log (2\delta) + C_2.
	\end{equation}
	Combining \eqref{Gpqestimate2} and \eqref{Gpqestimate3}
	we obtain
	\[ E_{\varphi}[\mu] \leq (r_0-1) \log \delta + 
	(r_0+2)\log 2  + C_1 + C_2. \]
	Since $\mu$ is a probability measure on $F_0 \subset F$, 
	we have
	$E_{\varphi}(F) \leq E_{\varphi}[\mu]$ and 
	\eqref{EFestimate1} follows.
	
	\medskip
	
	(b) The proof of part (b) is  similar. We take
	$0 < \delta < \delta_{\infty}/2$ and we assume that
	$F \in \mathcal F$ intersects $\partial U_{\delta}(p_\infty)$.
	Let $\widetilde{F}$ be a connected component of $F$
	that intersects $\partial U_{\delta}(p_\infty)$.
	
	Similar to \eqref{diamF0} we now find 
	a sub-continuum $F_0 \subset \widetilde{F} \cap \overline{U_{2\delta}(p_\infty)}$ with
	\[ \diam \left( z_{\infty}(F_0) \right) \geq \delta. \]
	Let $\omega$ be the equilibrium measure of $z_{\infty}(F_0)$, and $\mu$ its pullback under
	the mapping $p \mapsto z_{\infty}(p)$. Then
	\eqref{logestimate1} again holds, and from 
	\eqref{Gnearinfty} we  find for some constant $C_1$,
	\begin{align*} \iint G(p,q) d\mu(p) d\mu(q)
		& \leq \iint \log \frac{1}{|z_{\infty}(p)^{-1} - z_{\infty}(q)^{-1}|}
		d\mu(p) d\mu(q) + C_1 \\
		& = \iint \log \frac{1}{|z^{-1} - w^{-1}|} d\omega(z) d\omega(w) + C_1 \\
		& = \iint \log \frac{|z w|}{|z-w|} d\omega(z) d\omega(w) + C_1 \\
		& \leq - \log \frac{\delta}{4} + 2 \int \log |z| d\omega(z)
		+ C_1.
	\end{align*}
	Since $|z| \leq 2 \delta$ for $z \in \supp(\omega)$ we
	obtain 
	\begin{align} \nonumber \iint G(p,q) d\mu(p) d\mu(q)
		& \leq - \log \delta + \log 4 + 2 \log (2\delta) + C_1 \\
		& = \log \delta + 4 \log 2 + C_1, 
		\label{Gpqestimate4} \end{align}
	which is the analogue of \eqref{Gpqestimate2}, but note
	the different sign with $\log \delta$.
	
	Because of \eqref{phiatpinfty} there is $C_2$
	such that $\varphi(p) \leq r_{\infty} \log |z_\infty(p)| + C_2$
	for $p \in U_{\delta_0}(p_\infty)$,
	which gives us the analogue of \eqref{Gpqestimate3} 
	\begin{equation} \label{Gpqestimate5} 
		\int \varphi(p) d\mu(p) \leq r_{\infty} \log (2\delta) + C_2 
	\end{equation}
	since $\mu$ is supported on $U_{2\delta}(p_{\infty}) \subset
	U_{\delta_0}(p_\infty)$ and $z_{\infty}(p) \leq 2 \delta$
	for $p \in \supp(\mu)$. 
	
	Adding \eqref{Gpqestimate4} and \eqref{Gpqestimate5} we find
	\[ E_{\varphi}\left[\mu\right] \leq (1+r_{\infty}) \log \delta	+
	(r_{\infty} +4) \log 2 + C_1 + C_2. \]
	Since $\mu$ is a probability measure on $F$ we have
	$E_{\varphi}(F) \leq E_{\varphi}[\mu]$ and \eqref{EFestimate2}
	follows.
\end{proof}

\subsection{The max-min problem is well-posed}

If $r_0 > 1$ then it follows from \eqref{EFestimate1} that
$E_{\varphi}(F)$ is small (i.e., very negative) in case
$F$ comes very close to $p_{0}$. If $r_{\infty} > -1$
then similarly $E_{\varphi}(F)$ is small if $F$ comes close
to $p_{\infty}$.
This allows us to prove that the max-min energy problem is well-posed.

\begin{lemma}\label{lem:maxminwelldef}
	Suppose $r_0 > 1$ and $r_\infty > -1$. Then the max-min energy problem is well-posed: that is, $\sup_{F \in \mathcal F} E_\varphi(F)$ is finite.
\end{lemma}
\begin{proof}
	Let $m = \max(r_0-1, r_{\infty}+1) > 0$ and 
	let $C$ be as in Proposition~\ref{prop:residues}.	
	We fix $\delta = \delta_0/4$, so that by Proposition~\ref{prop:residues} (and the standing assumption $\delta_0 \leq \delta_\infty$),
	\begin{equation}\label{eq:upper bound well-def}
		E_\varphi(F) \leq m \log \delta + C < \infty
	\end{equation}
	for all $F \in \mathcal F$ that intersect $\partial U_\delta(p_0)$ or $\partial U_\delta(p_\infty)$.
	
	Now if $F \in \mathcal F$ does not intersect $\partial U_\delta(p_0)$ or $\partial U_\delta(p_\infty)$, any
	connected component $\widetilde{F}$ of $F$ 
	satisfies $\widetilde{F} \subset U_\delta(p_0)$, $\widetilde{F} \subset U_\delta(p_\infty)$ or $\widetilde{F} \subset X \setminus (U_\delta(p_0) \cup \partial U_\delta(p_\infty))$.
	In the first case, we can take a $\delta' < \delta$ for which 
	$\widetilde{F} \cap \partial U_{\delta'}(p_0) \neq \emptyset$ and apply Proposition~\ref{prop:residues} to conclude that
	\[
	E_\varphi(F) \leq (r_0 - 1) \log \delta' + C \leq m \log \delta + C.
	\]
	That is, the same upper bound as in \eqref{eq:upper bound well-def} can be used. 
	Similarly, 
	\[
	E_\varphi(F) \leq (r_\infty + 1) \log \delta' + C \leq m \log \delta + C
	\]
	in the second case and the upper bound from \eqref{eq:upper bound well-def} again works.
	It remains to find an upper bound for all $F \in \mathcal F$ with $F \subset X \setminus (U_\delta(p_0) \cup \partial U_\delta(p_\infty))$.
	
	Let $\Gamma$ be a simple curve connecting $p_0$ and $p_\infty$ without containing any further poles of $dV$.
	Then for all $\varepsilon > 0$ small enough, the fattened set $\Gamma_\varepsilon = \{p \in X : d(p,\Gamma) \leq \varepsilon\}$ does not contain any other poles of $dV$ either.
	We fix such a value of $\varepsilon$ so that $\varepsilon < \delta$ as well and define $K = \Gamma_\varepsilon \setminus (U_\delta(p_0) \cup U_\delta(p_\infty))$.
	The external field $\varphi$ is continuous and hence bounded on the compact set $K$, say $|\varphi| \leq M$ on $K$ for some $M \in \mathbb R$.
	
	As $K$ is compact, we can cover it with a finite number of coordinate charts $(U_j, z_j)$, $j = 1, \ldots, n$.
	By Lemma~\ref{lem:diameterestimate}, there are $\eta > 0$ and $C > 0$ such that for every subset $K'$ of $K$ with diameter at most $\eta$, we have $K' \subset U_j$ for some $j$ and $\diam(K') \leq C \diam (z_j(K'))$.
	Without loss of generality, we may assume $\eta \leq \varepsilon$.

	Now take $F \in \mathcal F$ with $F \subset X \setminus (U_\delta(p_0) \cup \partial U_\delta(p_\infty))$.
	Then $K \cap F$ contains a continuum $F_0$ with diameter $\eta$, namely a continuum that connects $\partial \Gamma_\eta$ with $\Gamma$.
	Hence there is a $j = 1,\ldots,n$ such that $F_0 \subset \tilde{U}_j$.
	Since $\diam(F_0) \leq C \diam(z_j(F_0))$, it follows that $z_j(F_0)$ is a continuum in $\mathbb C$ with (Euclidean) diameter at least $\frac\eta{C}$.
	
	Let $\omega$ be the equilibrium measure of $z_j(F_0)$ (in the usual logarithmic potential theory in the plane) without external field, and let $\mu$ be the pullback of $\omega$ by the local coordinate $z_j$.
	By similar arguments as given in the proof of Proposition~\ref{prop:residues}, we obtain
	\[
	\iint G(p,q) d \mu(p) d \mu(q) \leq -\log \left(\frac{\eta}{4 C}\right) + C_j = -\log \eta + \log 4 + \log C + C_j
	\]
	for some constant $C_j$ that only depends on $U_j$.
	Taking $\tilde{C} = \max_j C_j$ and using that $|\varphi| \leq M$ on $\supp(\mu) \subset K$, we thus find
	\begin{multline*}
		E_\varphi[\mu] = \iint G(p,q) d \mu(p) d \mu(q) + \int \varphi(p) d \mu(p) 
		\\ \leq -\log \eta + \log 4 + \log C + C_j + M \leq  -\log \eta + 2 \log 2 + \log C + \tilde{C} + M. 
	\end{multline*}
	Consequently, 
	\[
	E_\varphi(F) \leq E_\varphi(F_0) \leq E_\varphi[\mu] \leq -\log \eta + 2 \log 2 + \log C + \tilde{C} + M.
	\]
	As the upper bound is independent of $F$, this concludes the proof.
\end{proof}

\subsection{Continuity of the energy functional}

Rakhmanov \cite[Theorem~9.8]{Rak12} established the following fundamental result
for the complex plane. His arguments extend to the 
higher genus case and this was already done  by
Chirka \cite[Section 2.9]{Chi19} in the unweighted
case. We formulate the continuity
of the weighted energy functional for the case
of external field $\varphi = \Re V$ that is of interest
in the paper, but it holds for more general $\varphi$.

\begin{proposition} \label{prop:EFcontinuity}
	The weighted energy functional $F \mapsto E_\varphi(F)$
	is continuous on $\mathcal F$.
\end{proposition}

To prepare for the proof of Proposition~\ref{prop:EFcontinuity} we need three lemmas.
The first lemma provides an estimate for the usual 
Green's function $G_{\Omega}$ of 
$\Omega = X \setminus F$ with $F \in \mathcal F_{\delta}$. 
Recall that the Green's function $(p,q) \mapsto G_{\Omega}(p,q)$
is non-negative for $p,q \in X \times X$, zero whenever 
$p \in F$ or $q \in F$ and for a fixed $q \in \Omega$, 
$p \mapsto G_{\Omega}(p,q)$ is continuous on $X$, 
harmonic on $\Omega$, with  
\[ G_{\Omega}(p,q) = -\log |z(p)| + O(1)   \text{ as } p \to q \]
where $p \mapsto z(p)$ is a local coordinate around $q$.
The following estimate for the Green's function
is contained in \cite[Lemma~9.9]{Rak12} for compacts
in the complex plane. The extension to 
the higher genus case is mentioned in \cite[p.~331]{Chi19}.
\begin{lemma} \label{lem:lemma33}
	For every $0 < \delta < \delta_0$ there is a constant $C = C(\delta) > 0$
	such that for every $F \in \mathcal F_{\delta}$ and every $p \in X \setminus
	U_{\delta/2}(p_{\infty})$, one has
	\[ G_{\Omega}(p,p_{\infty}) \leq C \sqrt{d_H(p,F)}, \]
	where $\Omega = X \setminus F$.
\end{lemma}
\begin{proof}
	The argument of \cite[p.~331]{Chi19} applies
	since the diameter of each component of $F$ is
	at least $\eta > 0$ by Lemma~\ref{lem:prelim} (c). 
\end{proof}

Of course, it is possible that $F \in \mathcal F_\delta$ contains a pole of $dV$ where $\varphi \to +\infty$.
The next lemma shows that for our purposes, we may replace $\varphi$ by a continuous external field on $\mathcal F_\delta$.
For every $m \in \mathbb R$, we denote $\min(\varphi, m)$ by $\varphi_m$, which is a continuous function on $F$.

The proof is inspired by the proof of Theorem~I.1.3(b) in \cite{ST97}. 
However, in \cite{ST97}, the underlying set of the equilibrium problem is fixed, while here we have a family of sets.
Therefore, we decided to give a full proof.

\begin{lemma}\label{lem:boundextfield}
	For every $\delta < \delta_0$, there exists an $m \in \mathbb R$ such that for every $F \in \mathcal F_\delta$, the equilibrium measures of $F$ in the external field $\varphi$ and $\varphi_m = \min(\varphi,m)$ are the same.
	Moreover, for the shared equilibrium measure $\mu$, we have $\supp \mu \subset \{\varphi \leq m\}$ and $\varphi = \varphi_m$ on $\supp \mu$.
	Finally, $E_\varphi$ and $E_{\varphi_m}$ agree as functions on $\mathcal F_\delta$.
\end{lemma}
\begin{proof}	
	Let $\delta < \delta_0$. 
	Since the max-min energy problem is well-posed (see Lemma~\ref{lem:maxminwelldef}), the number $M = \sup_{F \in \mathcal F_\delta} E_\varphi(F)$ is finite.
	
	We write $K = X \setminus (U_\delta(p_0) \cup U_\delta(p_\infty))$ and note that every $F \in \mathcal F_\delta$ is contained in $K$.
	Because $G$ is bounded away from $-\infty$ 
	away from $p_\infty$ and the external field $\varphi$ is only $-\infty$ at $p_0$ and possibly at $p_\infty$ by assumption, we have
	\begin{equation*}
		m_G = \min_{p,q \in K} G(p,q) > -\infty \quad \text{ and } \quad m_\varphi = \min_{p \in K} \varphi(p) > -\infty.
	\end{equation*}
	We take $m \geq m_\varphi$ such that $m_G + \frac12 m + \frac12 m_\varphi \geq M+1$ and write $K_m = \{p \in K : \varphi(p) \leq m\}$.
	Note that $m$ is independent of $F \in \mathcal F_\delta$.
	
	Moreover, if $(p,q) \notin K_m \times K_m$, then $\frac12 (\varphi_m(p) + \varphi_m(q)) \geq \frac12(m+m_\varphi)$.
	Hence
	\begin{equation}\label{eq:Rak1}
		G(p,q) + \frac12(\varphi_m(p) + \varphi_m(q)) \geq m_G + \frac12 m + \frac12 m_\varphi \geq M+1, \quad (p,q) \notin K_m \times K_m.
	\end{equation}
	
	We now turn to the proof of the main statement. 
	Let $F \in \mathcal F_\delta$ and write $F_m = F \cap K_m = \{p \in F: \varphi(p) \leq m\}$.
	Take $\mu$ as the equilibrium measure of $F$ in the external field $\varphi_m$ and let $\mu^F$ be the equilibrium measure of $F$ in the external field $\varphi$.
	Then $E_{\varphi_m}[\mu] \leq E_{\varphi_m}[\mu^{F}] \leq E_\varphi[\mu^{F}] = E_\varphi(F) < M+1$, so that $\mu(F_m) > 0$ by \eqref{eq:Rak1}. 
	Hence $\tilde\mu = \mu|_{F_m}/\mu(F_m)$ is well-defined.
	Moreover, by \eqref{eq:Rak1},
	\begin{align*}
		E_{\varphi_m}[\mu] &= \left(\iint_{F_m \times F_m} + \iint_{(X \times X) \setminus (F_m \times F_m)}\right) [G(p,q) + \frac12 (\varphi_m(p) + \varphi_m(q))] d \mu(p) d \mu(q) \\
		&\geq \mu(F_m)^2 E_{\varphi_m}[\tilde\mu] + (M+1)(1-\mu(F_m)^2).
	\end{align*}
	Since $E_{\varphi_m}[\mu] < M+1$, the above leads to the contradiction $E_{\varphi_m}[\tilde\mu] < E_{\varphi_m}[\mu]$ unless $\mu(F_m) = 1$, in which case $E_{\varphi_m}[\mu] \geq E_{\varphi_m}[\tilde\mu]$.
	But then $\mu = \tilde\mu$ by the uniqueness of the equilibrium measure, so that $\supp(\mu) \subset F_m$.
	
	Because $\supp(\mu) \subset F_m$, we also have $\varphi = \varphi_m$ on $\supp(\mu)$.
	It directly follows that $E_{\varphi_m}[\mu] = E_\varphi[\mu]$. Since $\varphi_m \leq \varphi$ on $F$, we have 
	\[
	E_\varphi(F) \leq E_\varphi[\mu] = E_{\varphi_m}[\mu] = E_{\varphi_m}(F) \leq E_\varphi(F).
	\]
	Consequently, we have equality throughout the equation and the equilibrium measures of $F$ in the external fields $\varphi$ and $\varphi_m$ are the same.
	The final statement follows directly, which concludes the proof.
\end{proof}

The third lemma gives an estimate on the harmonic
extension. Suppose $\varphi$ is a continuous function
on $F$. Then there is a harmonic function $\widetilde{\varphi}$
on $X \setminus F$ with $\widetilde{\varphi} = \varphi$ on $F$.
It is simply the solution of the Dirichlet problem.
In our setting $\varphi$ itself is harmonic, wherever
it is finite.

Since $\varphi$ is not necessarily continuous on $F$, we replace $\varphi$ by $\varphi_m = \min(\varphi, m)$ from the lemma above.
The following is an estimate on $|\widetilde\varphi_m - \varphi_m|$
near $F$ that is due to Rakhmanov \cite[Lemma~9.7]{Rak12}
in the genus zero case (in a much more precise form, actually).
In the higher genus case we can follow the same proof 
as all the arguments are local in nature.
\begin{lemma} \label{lem:lemma34}
	For every $\delta < \delta_0$, $\varepsilon > 0$ and $m \in \mathbb R$,
	there is $\eta > 0$ such that
	for every $F \in F_{\delta}$ and every 
	$p \in X$ with $d_H(p,F) < \eta$ one has
	\[ \left| \widetilde\varphi_m(p) - \varphi_m(p)  \right| \leq \varepsilon, \]
	where $\widetilde\varphi_m$ is the harmonic extension of $\varphi_m = \min(\varphi,m)$ to $X \setminus F$.
\end{lemma}

\begin{proof}[Proof of Proposition~\ref{prop:EFcontinuity}.]
	
	It is enough to show that $F \mapsto E_{\varphi}(F)$
	is continuous on $\mathcal F_{\delta}$ for every $0 < \delta <
	\delta_0$.
	We fix $\delta$ and use Lemma~\ref{lem:boundextfield} to take an $m \in \mathbb R$ so that $E_{\varphi}(F) = E_{\varphi_m}(F)$ for all $F \in \mathcal F_\delta$.
	
	Let $\varepsilon > 0$. Then by Lemmas~\ref{lem:lemma33}
	and~\ref{lem:lemma34}
	there is $\eta > 0$ such that for any $F \in \mathcal F_{\delta}$ and $p \in X$ with $d_H(p,F) < \eta$
	we have $G_{\Omega}(p, p_{\infty}) \leq \frac{\varepsilon}{3}$ and $|\widetilde{\varphi}_m(p)-\varphi_m(p)| \leq \frac{\varepsilon}{3}$
	where $\Omega = X \setminus F$ and  $\widetilde{\varphi}_m$ is the
	harmonic extension of $\varphi_m$ from $F$ to $X \setminus F$.
	
	Now take $F_1, F_2 \in \mathcal F_\delta$ such that $d_H(F_1, F_2) < \eta$.
	Let $\mu_1 = \mu^{F_1}$ be the equilibrium measure of $F_1$ in the external field $\varphi$ and let $\mu_2$ be the balayage of  $\mu_1$ to $F_2$.
	Then $\mu_2$ is a probability measure on $F_2$ such that
	\[ \int \varphi_m d\mu_2 = \int \widetilde{\varphi}_m d\mu_1, \]
	where $\widetilde{\varphi}_m$ is the harmonic extension of
	$\varphi_m$ relative to $F_2$.
	Thus
	\begin{equation}\label{eq:Evarphicont2}
		\left| \int \varphi_m d \mu_2 - \int \varphi_m d \mu_1 \right| = \left| \int (\widetilde{\varphi}_m - \varphi_m) d \mu_1 \right| \leq \frac{\varepsilon}{3}
	\end{equation}
	since $|\widetilde{\varphi}_m - \varphi_m| \leq \frac{\varepsilon}{3}$ on $F_1$.
	
	Using the Green's function $G_\Omega$ for the open set $\Omega = X \setminus F_2$, it follows that
	\begin{multline*}
		\iint  G(p,q) d \mu_2(p) d \mu_2(q) \\
		= \iint G(p,q) d \mu_1(p) d \mu_1(q) - \iint G_\Omega(p,q) d \mu_1(p) d \mu_1(q) + 2 \int G_\Omega(p,p_\infty) d \mu_1(p) \\
		\leq \iint G(p,q) d \mu_1(p) d \mu_1(q) + 2 \int G_\Omega(p,p_\infty) d \mu_1(p),
	\end{multline*}
	where the first identity is the analogue of \cite[Lemma~9.6]{Rak12}.
	Hence
	\begin{equation}\label{eq:Evarphicont1}
		\iint G(p,q) d \mu_2(p) d \mu_2(q) \leq \iint G(p,q) d \mu_1(p) d \mu_1(q) + \frac{2 \varepsilon}{3}
	\end{equation}
	since $G_{\Omega}(p,p_{\infty}) \leq \frac{\varepsilon}{3}$
	for $p \in F_1$.
	Consequently, by \eqref{eq:Evarphicont1}, \eqref{eq:Evarphicont2} and Lemma~\ref{lem:boundextfield},
	\begin{align*} 
		E_\varphi(F_2) = E_{\varphi_m}(F_2) & \leq \iint G(p,q) d \mu_2(p) d \mu_2(q) + \int \varphi_m d \mu_2 \\ 
		& \leq \iint G(p,q) d \mu_1(p) d \mu_1(q) + \int \varphi_m d \mu_1 + \varepsilon \\
		& = E_{\varphi_m}(F_1) + \varepsilon = E_\varphi(F_1) + \varepsilon.
	\end{align*}
	By symmetry, we also have $E_\varphi(F_1) \leq E_\varphi(F_2) + \varepsilon$ and hence $|E_\varphi(F_1) - E_\varphi(F_2)| \leq \varepsilon$.
	This shows that $F \mapsto E_\varphi(F)$ is uniformly continuous on $\mathcal F_\delta$ and concludes the proof.
\end{proof}

\subsection{Proof of Theorem~\ref{thm:residues}}
Since an extremal set stays away from $p_0$ and $p_{\infty}$
in case both $r_0 > 1$ and $r_{\infty} > -1$, the continuity of the weighted energy functional allows us to prove 
Theorem~\ref{thm:residues}.

\begin{proof}[Proof of Theorem~\ref{thm:residues}.]
	Let $m = \max(r_0-1, r_{\infty}+1) > 0$ and 
	let $\delta_0$, $\delta_{\infty}$ and $C$ be as in  
	Proposition~\ref{prop:residues}.
	Pick an arbitrary $F_0 \in \mathcal F$, and take
	$\delta \in(0,1)$ with $2\delta < \min(\delta_0,\delta_{\infty})$
	small enough such that $F_0 \in \mathcal F_{\delta}$ 
	and $m \log \delta + C < E_{\varphi}(F_0)$.
	
	Since $\mathcal F_{\delta}$ is compact in the 
	Hausdorff distance, and $E_{\varphi}$ is continuous on 
	$\mathcal F_{\delta}$, there is $F \in \mathcal F_{\delta}$
	where $E_{\varphi}$ takes its maximum on $\mathcal F_{\delta}$.
	If $F' \in \mathcal F \setminus \mathcal{F}_{\delta}$, then
	either $F'$ intersects $\partial U_{\delta'}(p_0)$ for
	some $\delta' \leq \delta$, or $F'$ intersects 
	$\partial U_{\delta'}(p_{\infty})$ for some $\delta' \leq \delta$.
	In the first case, we find by part (a) of  Proposition~\ref{prop:residues} 
	\begin{multline*} E_{\varphi}(F') \leq (r_0-1) \leq \log \delta' + C
		\leq (r_0-1) \log \delta  + C \\
		\leq m \log \delta + C
		\leq E_{\varphi}(F_0) \leq E_{\varphi}(F). \end{multline*}
	Similarly, in the second case we use part (b) and
	we also find $E_{\varphi}(F') \leq E_{\varphi}(F)$.

The continuum $F$ satisfies the conditions of  Corollary~\ref{cor:maxmin} and the theorem follows. 
\end{proof}

\section{Proof of Theorem~\ref{theo:QDhigher}}
	\label{sec:critical-measures-and-quadratic-differentials}

\subsection{Preparation for the proof} \label{sec:genus1}
By Proposition~\ref{prop:CMhigher} a critical measure $\mu$ satisfies
\eqref{CMQD1} for every $C^1$ vector field $h$. 
As indicated in Section~\ref{sec:statement-of-results}, we use the $(2,-1)$-Cauchy kernel
defined in Proposition~\ref{prop:C21kernel} to extract relevant
information from \eqref{CMQD1}.

In the genus zero case one has \eqref{CM1} and $h$ is a function (not 
a vector field). For the proof of Theorem~\ref{theo:MFR} one
takes 
\begin{equation} \label{CMQD2}  
	h(s) = \frac{1}{z-s} \end{equation}
with a fixed $z \in \mathbb C \setminus \supp \mu$.
Then
\begin{equation} \label{CMQD3} 
	\frac{h(s) - h(t)}{s-t} = \frac{1}{(z-s)(z-t)} \end{equation}
and
\begin{equation} \label{CMQD4} 
	\iint \frac{h(s)-h(t)}{s-t} d\mu(s) d\mu(t) =
	\left[\int \frac{d\mu(s)}{z-s} \right]^2 \end{equation}
which is a crucial step in the proof of Theorem~\ref{theo:MFR}, see e.g. \cite[Proof of Lemma~5.1]{MFR11} or \cite[Proof of Proposition~3.7]{KS15}.

We assume that $X$ has the form \eqref{Xgenus1}  as discussed in Section~\ref{sec:main-result} with bounded real oval $C_1$.
We have an analogue of \eqref{CMQD4} in case $\mu$ is invariant
under the involution \eqref{involution} and its support does
not intersect $C_1$.

\begin{proposition} \label{prop:CMQD} 
	Suppose $\mu$ is a compactly supported measure 
	on $X \setminus ( \{p_\infty\} \cup C_1)$ that is 
	invariant under the involution $\sigma$ from \eqref{involution}. 
	Then there exists $a \in  C_1$ 	such that
	\begin{multline}  \label{CMQD6} 
		\iint \left( C(p,q) C^{(2,-1)}(u,p;a)
		+ C(q,p) C^{(2,-1)}(u,q;a)  - C(u,p) C(u,q) \right) \\
		d\mu(p) d\mu(q) 
		= 0, \qquad u \in X.
	\end{multline}
\end{proposition}
\begin{proof}
	We compare the expressions
	\begin{equation} \label{CMQD8} 
		C(p,q) C^{(2,-1)}(u,p;a)
		+ C(q,p) C^{(2,-1)}(u,q;a) \end{equation}
	and
	\begin{equation} \label{CMQD9} C(u,p) C(u,q). \end{equation}
	Both \eqref{CMQD8} and \eqref{CMQD9} are, for fixed $p,q$, and $a$,  
	meromorphic  quadratic differentials in $u$,
	with simple poles in $u=p$, $u=q$ and a double pole at $u=p_{\infty}$. 
	The residue at $u=p$
	is $C(p,q)$ for both expressions, and the residue
	at $u=q$
	is $C(q,p)$, again the same for both of them. These two
	poles disappear if we take the difference, and the only
	pole is the double pole at $u=p_{\infty}$.
	Then the left-hand side of \eqref{CMQD6} 
	is also a meromorphic quadratic differential in $u$
	with a possible double pole at $u=p_{\infty}$ only.
	
	\medskip
	Suppose $a \in X \setminus (\{p_{\infty}\} \cup \supp(\mu))$ 
	is a zero of the meromorphic  differential
	$\int C(u,q) d\mu(q)$. Then it is a double zero
	of $\iint C(u,p) C(u,q) d\mu(p) d\mu(q) = \left(\int C(u,p)d \mu(p)\right)^2$ and since both
	$C^{(2,-1)}(u,p;a)$ and $C^{(2,-1)}(u,q;a)$ have a double zero 
	at $u=a$ as well, see Proposition~\ref{prop:C21kernel} (b), 
	we find
	that  \eqref{CMQD6} has a double zero at $u=a$.
	It has at most a double pole at $u=p_{\infty}$, and there are
	no other poles and zeros (since $X$ has genus one), unless
	it vanishes identically.
	The only non-zero meromorphic quadratic differentials on $X$ that have double zeros and a double pole at $p_\infty$ are
	$ z \frac{dz^2}{w^2}$, $(z-z_1) \frac{dz^2}{w^2}$, $(z-z_2) \frac{dz^2}{w^2}$
	and their scalar multiples. Thus if  $a \not\in \{p_0,p_1,p_2\}$
	where  $p_0 = (0,0)$, $p_1 = (0,z_1)$, $p_2 = (0,z_2)$
	then \eqref{CMQD6} holds. 
	
	\medskip
	We now show that the differential $\int C(u,q) d \mu(q)$ has exactly two zeros on $C_1$: if at least one of them is not at the branch point, then the proof follows by the above discussion; we then will consider the special case where both zeros are at a  branch point.
	Now, let $\mu$ be as in the statement of the proposition. 
	Then 
	\[ G^\mu(u) := \int G(u,q) d\mu(q) \] 
	is real analytic on the cycle $C_1$,
	since the support of $\mu$ is disjoint from $C_1$.
	Since $\mu$ is invariant under $\sigma$, we also have
	$G^\mu(\sigma(u)) = G^\mu(u) $
	which implies that
	$\frac{\partial}{\partial y} G^\mu(u)  = 0$ for $u \in C_1$, 
	if we use $z=x+iy$ as the local coordinate at a point
	$u = (w,z)$ on $C_1$.
	By compactness, $G^\mu$ attains a maximum and a minimum on
	$C_1$. If an extremum is attained at $a\in C_1$, 
	then we also have $\frac{\partial}{\partial x} G^\mu(u) = 0$
	at $u = a$. 
	Hence by \eqref{Cauchy10} we have that $a$ is a zero of
	\[ \int C(u,q)d\mu(q) = - \left(\partial_x - i \partial_y \right) 
	\int G(u,q) d\mu(q). \] 
	Thus, by what we already proved, if $G^\mu$ attains an extremum 
	on $C_1$  at a point $a \not\in \{p_1,p_2\}$ then \eqref{CMQD6} holds. 
	
	\medskip
	Next, we consider the special case $\mu = \delta_q$
	with $q \in C_2 \setminus \{p_0, p_{\infty}\}$ on the unbounded cycle. 
	On the complex torus the extrema of $G(u,q)$ are 
	attained at the zeros of
	\begin{equation} \label{CMQD10} 
		u \mapsto \frac{\theta_1'(u-v)}{\theta_1(u-v)} 
		-\frac{\theta_1'(u)}{\theta_1(u)}
	\end{equation}
	where $v$ is the image of $q$ under the Abel map,
	see \eqref{Cauchy1} with $\Im v = 0$.
	The elliptic function \eqref{CMQD10} has two simple poles 
	at $u=0$ and $u=v$. By Abel's theorem the two zeros 
	add up to $v$ (modulo $\Lambda = \mathbb Z + \tau \mathbb Z$). 
	Since $v \neq \frac{1}{2}$, (since $v= \frac{1}{2}$ corresponds
	to $q = p_0$ in $X$), the two zeros cannot be 
	$\frac{1}{2} \tau$ and 
	$\frac{1}{2} + \frac{1}{2} \tau$.
	On $X$ this means that $C(u,q)$ does not vanish at both
	$p_1$ and $p_2$. 
	\medskip
	
	Finally, consider the case that $G^\mu$ attains its extrema
	on $C_1$ only at the branch points $p_1$ and $p_2$, so that 
	$\int C(u,p) d\mu(p)$ vanishes at both $p_1$ and $p_2$. 
	Pick $q$ on the unbounded cycle, $q \neq p_0$, $q \neq p_\infty$
	and consider $\mu_t = \mu + t \delta_{q}$.
	Since $C(u,q)$ does not vanish at both $p_1$ and $p_2$,
	we find that for $t > 0$,
	\[ \int C(u,p) d\mu_t(p) =  \int C(u,p) d\mu(p) + t C(u,q) \]
	is not zero at both $p_1, p_2$, 
	and therefore it will have a zero somewhere else on the cycle.
	The zero depends on $t>0$, say $a_t \in C_1$, and 
	the identity \eqref{CMQD6} holds for $\mu_t$ and $a_t$. 
	Letting $t \to 0+$ and by using a continuity and compactness argument,  
	we find that \eqref{CMQD6} holds for $t=0$
	as well, where $a\in C_1$ is any limit point 
	of $(a_t)_{t > 0}$ as $t \to 0+$.
\end{proof}

If $\int |C(u,q)| d\mu(q) < \infty$, which is the case
for $u$-a.e.\ on $X$,  then \eqref{CMQD6} can be rewritten to
\begin{multline} \label{CMQD7} 
	\iint \left( C(p,q) C^{(2,-1)}(u,p;a)
	+ C(q,p) C^{(2,-1)}(u,q;a) \right) d\mu(p) d\mu(q) \\
	= \left[ \int C(u,q) d\mu(q) \right]^2.
\end{multline}
The identity \eqref{CMQD7} can be viewed as a genus one analogue
of \eqref{CMQD4}. Note however that we do not have an analogue of the divided difference identity \eqref{CMQD3}.

\subsection{Proof of Theorem~\ref{theo:QDhigher}}\label{sec:proof-of-main-result}

\begin{proof}
	Suppose $\mu$ is a critical measure in the external
	field $\varphi = \Re V$ that is invariant under
	the involution $\sigma$. Suppose $\supp(\mu) \subset X \setminus (\{p_{\infty}\} \cup C_1)$.  Let $a \in C_1$ be as in Proposition~\ref{prop:CMQD}. 
	
	We take $h(p) = C^{(2,-1)}(u,p;a)$ which is a $C^1$
	vector field with a double pole at $p=a$ and a simple
	pole at $p=u$. Thus $h$ is a $C^1$ vector field
	on $\supp(\mu)$ if $u \not\in \supp(\mu)$, as it is already
	assumed that $a \not\in \supp(\mu)$. Thus by 
	Proposition~\ref{prop:CMhigher}~(b) we have 
	$D_{V,h}(\mu) = 0$ provided that $u \not\in \supp(\mu)$.
	
	With an approximation argument as in \cite[Lemma~3.5]{KS15}
	we can extend the equality $D_{V,h}(\mu) = 0$ from
	$u \not\in \supp(\mu)$  to any $u \in X \setminus \{p_{\infty}\}$ for which 
	\[ \int \frac{d\mu(q)}{d(u,q)} < +\infty. \]
	which holds for a.e.\ $u$ on $X$. 
	
	From \eqref{CMhigher3} we thus have
	\begin{multline} \label{QDhigher3} 
		\iint \left( C(p,q) C^{(2,-1)}(u,p;a) 
		+ C(q,p) C^{(2,-1)}(u,q;a) \right) d\mu(p) d\mu(q) \\
		= \int  C^{(2,-1)}(u,q;a) dV(q) d\mu(q),
		\quad \text{$u$-a.e.\ on $X$.} \end{multline}
	Using \eqref{CMQD6} we get  from \eqref{QDhigher3} that
	\begin{equation} \label{QDhigher4} 
		\left[\int C(u,q) d\mu(q) \right]^2  
		= \int  C^{(2,-1)}(u,q;a) dV(q) d\mu(q),
		\quad \text{$u$-a.e.\ on $X$,} \end{equation}
	which is easily seen to be \eqref{QDhigher2} with $\upomega$
	given by \eqref{QDhigher3}. 
	
	Note that $\upomega$ is indeed a meromorphic quadratic differential on $X$,
	since the pole at $u=q$ cancels out in the difference
	\[ C(u,q) dV(u) - C^{(2,-1)}(u,q;a) dV(q) \]
	and therefore the integral transform in \eqref{QDhigher1} indeed extends analytically across the support $\Sigma$ of the measure $\mu$.
	
	\medskip
	For the proof of (b) we can follow the proof of Proposition 3.8 in \cite[Proposition 3.8]{KS15}, since all arguments are local in nature.
	
	\medskip
	
	From \eqref{QDhigher1} and working in a local coordinate
	around $u \in \Sigma$, we have
	\[ \left(\int C(u,q) d\mu(q) - \frac{dV(u)}{2}\right)_{\pm}
	= \pm Q(u)^{1/2} du \]
	where $\pm$ denote the boundary limits on $\Sigma$.
	Thus
	\[ \left(\int C(u,q) d\mu(q) - \frac{dV(u)}{2}\right)_{+}
	= - \left(\int C(u,q) d\mu(q) - \frac{dV(u)}{2}\right)_{-} \]
	or put otherwise
	\[ \left(\int C(u,q) d\mu(q) \right)_+ + \left(\int C(u,q) d\mu(q) \right)_-
	- dV(u) = 0 \qquad \text{ on } \Sigma \]
	as $dV$ does not have a jump on $\Sigma$.
	If $g(u) = \int \mathcal G(u,p) d\mu(p)$, then
	\[ g_+(u) + g_-(u) + V(u) = \text{const} \quad \text{ on } D \cap \Sigma. \]
	Taking the real part we obtain part (c). Taking
	the imaginary part and using the Cauchy-Riemann equations,
	we obtain part (d).  
\end{proof}

\appendix

\section{Existence and symmetry of the bipolar Green's function}

Let $X$ be a Riemann surface and let $p_\infty$ be a distinguished point at infinity.
Our goal is to prove Proposition~\ref{prop:bipolarGreen}
and in particular part (d) that gives us the symmetry
$G(p,q) = G(q,p)$ of the bipolar Green's function.
For this proof we rely on Riemannian geometry
where the existence of a symmetric Green's
function for the Laplacian is known. 

\subsection{Riemann surfaces as Riemannian manifolds}\label{sec:riemannianmanifold}

We first recall some relevant notions from 
Riemannian geometry.
The main reference in this section is \cite{Jost06}.
As shown in Lemma~2.3.3 in \cite{Jost06}, $X$ admits a conformal Riemannian metric $\rho$, which is given in the local coordinate $(U,z)$ by
\[
\rho_U(z)^2 dz d\overline{z}, \quad \rho_U(z) > 0,
\]
where $\rho_U$ is smooth and which transforms correctly under a holomorphic change of local coordinates.
This turns $X$ into a Riemannian manifold.
The Riemannian metric allows for the definition of the length of a curve $\gamma \subset U$ and area of a measurable set $B \subset U$, namely
\[
\ell(\gamma) := \int_\gamma \rho_U(z) |dz|, \quad \operatorname{area}(B) := \frac{i}{2} \int_B \rho_U(z)^2 dz d\overline{z};
\]
see also \cite[p. 21]{Jost06}.
The length of general curves is computed by splitting the curve into a final number of pieces, each of which is contained in a single coordinate chart; a similar method works for the area.
Note that $\operatorname{area}(X)$ is finite as $X$ is compact.

The Riemannian metric $\rho$ globally defines an area form (i.e.,  a nowhere-vanishing $2$ form) $dA$ (the dependence on $\rho$ is suppressed in the notation) in the local coordinate $(U,z)$ by
\[
\frac{i}{2} \rho_U(z)^2 dz d\overline{z}.	
\]
The area form $dA$ will be used to integrate functions.

The distance $d : X \times X \to [0,+\infty)$ between two points $p$ and $q$ can then be defined as
\begin{equation}\label{eq:defRiemDistance}
	d(p,q) := \inf\left\{\ell(\gamma) \mid \text{$\gamma: [0,1] \to X$ is a curve with $\gamma(0) = p$ and $\gamma(1) = q$} \right\}.
\end{equation}
The metric topology on $X$ defined by $d$ coincides with the original topology on $X$. 
This is because in a coordinate chart $(U,z)$, the induced distance $d(z(p),z(q))$ is equivalent to the Euclidean distance in $z(U)$ in the sense that for a fixed compact set $K \subset U$,
\begin{equation}\label{eq:distanceequiv}
	c |z(p) - z(q)| \leq d(z(p),z(q)) \leq C |z(p) - z(q)|, \quad p,q \in K,
\end{equation}
where $c > 0$ and $C > 0$ are the minimum and maximum of $\rho_U$ on $K$ respectively.
See also \cite[Theorem~1.18]{Aub98}.

\begin{lemma}\label{lem:diameterestimate}
	Let $K \subset X$ be compact and let $\{U_j\}_{j=1}^n$ be a finite open cover of $K$ with coordinate charts $(U_j, z_j)$. Then there exists an $\eta > 0$ and a $C > 0$ such that for every subset $K'$ of $K$ with diameter at most $\eta$, $K' \subset U_j$ for some $j$ and
	\[
	\diam(K') \leq C \diam(z_j(K')).
	\]
\end{lemma}

Note that we have the diameter with respect to $d$ (see \eqref{eq:defRiemDistance}) on the left and the standard Euclidean diameter on the right.

\begin{proof}[Proof of Lemma~\ref{lem:diameterestimate}]
	Let $K \subset X$ be compact and let $\{U_j\}_{j=1}^n$ be a finite open cover of $K$ with coordinate charts $(U_j, z_j)$ be given.
	By the properties of the metric topology, we can find open sets $\tilde{U}_j$ that are compactly contained in $U_j$ (that is, the closure of $\tilde{U}_j$ is compact and contained in $U_j$) such that $K \subset 
	\cup_{j=1}^n \tilde{U}_j$.
	Hence by \eqref{eq:distanceequiv}, for every $j = 1, \ldots, n$, there are numbers $C_j > 0$ such that
	\[
	d(p,q) \leq C_j |z_j(p) - z_j(q)|, \quad p,q \in \tilde{U}_j.
	\]

	Let $\eta > 0$ be a Lebesgue number for the open cover $\{\tilde{U}_j\}_{j=1}^n$ of $K$ (so every subset of $K$ with diameter at most $\eta$ is contained in some $\tilde{U}_j$) and take $C = \max_j C_j > 0$.
	Suppose that $K' \subset K$ has diameter $\diam(K') \leq \eta$. 
	Then $K'$ lies in some $\tilde{U}_j$.
	Moreover,
	\begin{align*}
		\diam(K') &= \sup_{p,q \in K'} d(p,q) \leq C_j \sup_{p,q \in K'} |z_j(p) - z_j(q)| \\ 
		&\leq C \sup_{p,q \in K'} |z_j(p) - z_j(q)| = C \diam(z_j(K')),
	\end{align*}
	which concludes the proof.
\end{proof}

The Riemannian metric $\rho$ moreover defines the Laplace-Beltrami operator $\Delta$ on functions $f \in C^2(X)$ by locally setting
\[
\Delta f = \frac{4}{\rho_U(z)^2} \frac{\partial}{\partial z} \frac{\partial}{\partial \overline{z}} f,
\]
see \cite[Definition 2.3.3]{Jost06} and also \cite[Section 1.1]{Chi18}.
Note that $\Delta f$ is a function on $X$. Moreover,
\[
\Delta f = - \star d \star df 
\]
where $\star$ denotes the Hodge star operator on $k$-forms (see \cite[Section 5.2]{Jost06}).
The $2$-form $d \star df$ is independent of $\rho$ (and is sometimes taken as a definition for the Laplacian, see e.g.\ \cite{FK80}) and hence $\Delta$ only depends on $\rho$ through the application of the Hodge star operator on $d \star d f$. 

\subsection{Proof of Proposition~\ref{prop:bipolarGreen}}

Because $X$ can be turned into a compact Riemannian manifold, it carries a Green's function of the Laplacian \cite[Theorem~4.13]{Aub98} (see also \cite[Theorem~2.1]{Bel19}).
This is a real-valued function $\tilde{G}$ defined on $X \times X$ minus the diagonal that is smooth, symmetric \begin{equation} \label{eq:Gtildesymmetric}
	\tilde{G}(p,q) = \tilde{G}(q,p)
	\qquad \text{ for  } p \neq q, \end{equation}
and satisfies the distributional identity
\begin{align}\label{eq:weak Laplacian Green's function Laplacian}
	\Delta_p \tilde{G}(p,q) = \delta_q(p) - \operatorname{area}(X)^{-1},
\end{align}
that is,
\[
\int \tilde{G}(p,q) \Delta f(p) dA(p) = f(q) - \operatorname{area}(X)^{-1} \int f dA
\]
for all $C^2$ functions $f$.
Moreover, \eqref{eq:Gtildesymmetric} and \eqref{eq:weak Laplacian Green's function Laplacian} 
define $\tilde{G}$ uniquely up to an additive constant.

Furthermore, it has the following local behavior:
if $z$ is a local coordinate around a point $p_0 \in X$, then
\begin{align} \label{eq:Gtildelocal}
	\tilde{G}(p,q) = -\frac{1}{2 \pi} \log |z(p) - z(q)| + O(1)
\end{align}
uniformly for $p$ and $q$ in a neighborhood of $p_0$,
as follows e.g.\ from the proof of Theorem~4.13(c) in \cite{Aub98} combined with \eqref{eq:distanceequiv}.

It should be noted that $\tilde{G}$ does not have any special behavior at $p_\infty$. 
From $\tilde{G}$ we obtain the bipolar Green's function 
with pole at $p_\infty$ as follows.

\begin{proposition}
	The function defined by 
	\begin{equation} \label{eq:Gdef}
		G(p,q) = 2 \pi [\tilde{G}(p,q) - \tilde{G}(p,p_\infty) - \tilde{G}(q,p_\infty)]
	\end{equation}
	satisfies the properties stated in Proposition~\ref{prop:bipolarGreen}.
\end{proposition}
\begin{proof}	
	Parts (b) and (c) of Proposition~\ref{prop:bipolarGreen}  follow directly from
	\eqref{eq:Gtildelocal} and \eqref{eq:Gdef} and part (d)
	follows from \eqref{eq:Gdef} and the symmetry \eqref{eq:Gtildesymmetric} 
	of $\tilde{G}$. Hence it remains to check part (a).
	
	Fix a $q \in X \setminus \{p_\infty\}$. 
	The function $p \mapsto G(p,q)$ is clearly real-valued
	on $X \setminus \{ p_{\infty}, q\}$.
	Moreover, it follows from \eqref{eq:weak Laplacian Green's function Laplacian} and \eqref{eq:Gdef} that
	\begin{equation*}
		\Delta_p G(p,q) = 2 \pi (\delta_q(p) - \delta_{p_\infty}(p))
	\end{equation*}
	in a distributional sense.
	The right-hand side is zero for $p \notin \{p_\infty,q\}$, hence $p \mapsto G(p,q)$ is weakly harmonic on $X \setminus \{p_\infty,q\}$.
	By Weyl's lemma (see e.g. \cite[Theorem~3.4.2]{Jost06}), the function $p \mapsto G(p,q)$ is then harmonic on $X \setminus \{p_\infty, q\}$.
	This concludes the proof.	
\end{proof}

\end{document}